\documentclass[a4paper,11pt]{LTD_paper}
\makeatletter

\usepackage{amsmath}
\usepackage{mathtools}
\usepackage{amsthm}
\usepackage{amsfonts}  
\usepackage{amssymb}
\usepackage{mathrsfs}
\makeatother
\usepackage{leftidx}
\usepackage[utf8]{inputenc}
\usepackage{graphicx}
\usepackage{hyperref}
\usepackage[backend=biber,
			style=numeric,
                sorting=nyt,
                sortcites=true,
			isbn=false,
			url=false,
			eprint=false,
                maxbibnames=5,
                giveninits=true,
                autopunct=true,
                autolang=hyphen,
                hyperref=true,
                abbreviate=true,
                backref=true,
                citestyle=ieee,backref=false]{biblatex}
\AtBeginBibliography{\footnotesize}
\DeclareNameAlias{default}{family-given}

\let\textquotedbl="
\usepackage{tikz-cd}
\usepackage{tikz} 
\tikzstyle{every pin}=[
					   pin edge = black,					   
					   font=\footnotesize]
\tikzstyle{block} = [draw, fill=blue!20, rectangle, minimum height=3em, minimum width=6em]
\tikzstyle{sum} = [draw, fill=blue!20, circle, node distance=1cm]
\tikzstyle{input} = [coordinate]
\tikzstyle{output} = [coordinate]
\tikzstyle{pinstyle} = [pin edge={to-,thin,black}]

\ltdsetup{\today}{}{}

\usepackage{caption} 
\usepackage{subcaption} 
\usepackage{tabularx} 
\usepackage[font=small]{caption}
\usepackage[utf8]{inputenc}
\usepackage{paralist}

\newtheorem{theorem}{Theorem}[section]
\newtheorem{definition}[theorem]{Definition}

\newtheorem{proposition}[theorem]{Proposition}
\newtheorem{example}[theorem]{Example}
\newtheorem{remark}[theorem]{Remark}

\title{
A new Lagrangian approach to optimal control of second-order systems
}
\date{\today}

\author
{Michael Konopik\footnote{Friedrich-Alexander-Universität Erlangen-Nürnberg (FAU), Institute of Applied Dynamics (LTD), Immerwahrstrasse 1, 91058 Erlangen, Germany. Email: \href{mailto:michael.konopik@fau.de}{michael.konopik@fau.de}}\ \thanks{The work of this author has been supported by Deutsche Forschungsgemeinschaft (DFG), Grant No. LE 1841/12-1, AOBJ: 692092.}
\qquad
Sigrid Leyendecker\footnote{Friedrich-Alexander-Universität Erlangen-Nürnberg (FAU), Institute of Applied Dynamics (LTD), Immerwahrstrasse 1, 91058 Erlangen, Germany. Email: \href{mailto:sigrid.leyendecker@fau.de}{sigrid.leyendecker@fau.de}}\ 
\qquad 
 Sofya Maslovskaya\footnote{\textit{First author}. Universität Paderborn (UPB), Numerical Mathematics and Control (NMC), Warburger Straße 100, 33098 Paderborn, Germany. Email: \href{mailto:sofya.maslovskaya@upb.de}{sofya.maslovskaya@upb.de}}\ \thanks{The work of this author has been supported by Deutsche Forschungsgemeinschaft (DFG), Grant No. OB 368/5-1, AOBJ: 692093}
\\
Sina Ober-Bl\"obaum\footnote{  Universität Paderborn (UPB), Numerical Mathematics and Control (NMC), Warburger Straße 100, 33098 Paderborn, Germany. Email: \href{mailto:sinaober@math.uni-paderborn.de}{sinaober@math.uni-paderborn.de}}\
\qquad
Rodrigo T.~Sato Mart{\'\i}n de Almagro\footnote{\textit{First author}, \textit{corresponding author}. Friedrich-Alexander-Universität Erlangen-Nürnberg (FAU), Institute of Applied Dynamics (LTD), Immerwahrstrasse 1, 91058 Erlangen, Germany. Email: \href{mailto:rodrigo.t.sato@fau.de}{rodrigo.t.sato@fau.de}}\
}

\begin{document}
\maketitle



\section*{Abstract}
In this work, we propose and study a new approach to formulate the optimal control problem of second-order differential equations, with a particular interest in those derived from force-controlled Lagrangian systems. The formulation results in a new hyperregular control Langrangian and, thus, a new control Hamiltonian whose equations of motion provide necessary optimality conditions.
We compare this approach to Pontryagin's maximum principle (PMP) in this setting, providing geometric insight into their relation. This leads us to define an extended Tulczyjew's triple with controls. Moreover, we study the relationship between Noether symmetries of this new formulation and those of the PMP.


    \vspace{2mm}

    \textbf{\bfseries \sffamily Keywords:} \textit{optimal control},
    \textit{Pontryagin's maximum principle},
    \textit{second-order differential equations},
    \textit{Lagrangian systems}, \textit{Hamiltonian systems}, \textit{necessary optimality conditions}, \textit{calculus of variations},  \textit{symplectic geometry},
    \textit{symmetries}.

    \vspace{2mm}
    
\textbf{\bfseries \sffamily Mathematics Subject Classification:} 
49K15,
34H05,
53Zxx,
53D05,
70G65,
70Hxx,
70H15,
70H50,
70H33,
70Q05.


\section{Introduction}


    Optimal control problems (OCPs), in particular for mechanical systems, possess a rich internal geometric structure. Generically, one has a controlled dynamical system, the state dynamics, and a performance index, a cost function. An OCP links this state dynamics with an adjoint or costate dynamics, living in the cotangent space over the state dynamics, through this cost. This endows the problem with a presymplectic structure \cite{EcheverriaEnriquez03}, as Pontryagin's maximum principle (PMP) illustrates in its explicitly Hamiltonian formulation \cite{pontryagin1962,Agrachev91}. At optima, the joint state-adjoint dynamics is symplectic. This structure has analytical and numerical consequences which make it worthwhile studying and preserving.\\

    Whenever the underlying state dynamics possesses some additional structure, the OCP inherits this structure too. In the case of mechanical systems, one typically has a preexisting structure, e.g.~symplectic or Poisson \cite{AbrahamMarsden,ArnoldMethods}, though usually the addition of control forces destroys this structure. Nevertheless, if these forces are not so strong as to completely overpower the conservative part of the dynamics, it can be argued that the original structure is still important for the overall behavior of the system. In the case of regular forced Lagrangian systems \cite{MarsdenWest01}, whose preexisting structure is symplectic, numerical results seem to corroborate this argument, favoring the use of symplectic, and more so variational, integrators in that case.\\

    It is because of this structure, and the depth and deceptive simplicity provided by the theory of variational integrators that pushed us to pursue a Lagrangian approach to OCPs in \cite{Leyendecker2024new}. This work is intended as a natural continuation and generalization of that article. In that work, we introduced a new Lagrangian approach for optimal control problems of a particular form, namely, those with running cost functions quadratic in the controls, and affine-controlled systems.\\

    In this work, we discuss this new approach more broadly, for general running costs and controlled second-order dynamics. We also treat the case where the given dynamics comes from a force-controlled Lagrangian system. Moreover, we study the relationship between this formulation and PMP. One of our main results shows that both the resulting new control Lagrangian and Hamiltonian formulations are equivalent to PMP and the key to this relation is Tulczyjew's triple \cite{TulczyjewLag,TulczyjewHam}. Examples are provided throughout the text to highlight and clarify some of the content.
    

    We begin with a general overview of OCPs and their variational treatment, introducing the concept of algebraic regularity, superregularity and hyperregularity. Additionally, we present the concept of controlled second-order differential equation (SODE) from a geometric perspective and define the concepts of full and under-actuation.\\
    
    In Section \ref{sec:ocp_sodes}, we proceed to explore the case of OCPs for controlled SODEs, introducing our new Lagrangian approach and proving its equivalence with standard augmented approaches. We also briefly present an alternative Lagrangian approach that leads to higher-order Lagrangians \cite{Colombo2016} and compare it with ours. The interpretation of the boundary terms in our
formulation 
    is discussed, along with their connection to the theory of generating functions. Next, we reformulate our approach in the case of regular force-controlled Euler-Lagrange equations and begin examining how it relates to Pontryagin's approach for the corresponding force-controlled Hamiltonian equations.\\

    We postpone most of the description and analysis of the geometric spaces where our approach takes place to Section \ref{sec:tulcyjew_and_PMP}. In it, we discuss how the formulation is inextricably related to Tulczyjew's triple and proceed to construct the necessary spaces to include the controls. This allows us to show the relationship between our formulations and Pontryagin's Hamiltonian in an intrinsic way.\\

    Finally, in Section \ref{sec:symmetries_noether} we discuss the invariance properties of our new Lagrangian and its consequences using Noether's theorem.


\subsection{Optimal control problems from a geometric and variational point of view}

    Consider a vector bundle $(\mathcal{E}, \pi^{\mathcal{E}}, \mathcal{M})$, where the base space $\mathcal{M}$ is a smooth manifold and the typical fiber of $\mathcal{E}$ is the vector space $\mathcal{N}$ \cite{Saunders89,Lee13}. Let us assume that $n = \dim \mathcal{N} \leq \dim \mathcal{M} = m$. In this context, we refer to $\mathcal{M}$ as the state space, $\mathcal{N}$ as the control space and $\mathcal{E}$ as the state-control space. Thus, in local adapted coordinates, a point in $\mathcal{E}$ can be written as $(x^1,...,x^m,u^1,...,u^n)$, or $(x,u)$ in short, so that $\pi^{\mathcal{E}}(x,u) = x$. A time-continuous (and time-independent) controlled system on $\mathcal{M}$ with state-control space $\mathcal{E}$ is defined by a bundle morphism over $\mathcal{M}$, $f : \mathcal{E} \to T\mathcal{M}$, with $(T \mathcal{M}, \tau_{\mathcal{M}}, \mathcal{M})$ the tangent bundle of $\mathcal{M}$. A curve $\sigma: [0,T] \to \mathcal{E}$, with $T \in \mathbb{R^+}$, generates a controlled trajectory $\gamma = \pi^{\mathcal{E}} \circ \sigma$ if and only if
    \begin{equation*}
        \frac{d}{d t} \gamma(t) = f(\sigma(t)), \quad \text{for almost all } t \in [0,T].
    \end{equation*}

    Consider an initial submanifold $\mathcal{M}_0 \subset \mathcal{M}$, and assume $\mathcal{M}_T \subset \mathcal{A}(T,\mathcal{M}_0)$, where $\mathcal{A}(T,\mathcal{M}_0)$ denotes the reachable set at time $T$ starting from $\mathcal{M}_0$ \cite{Boscain2019}. Next, consider the set
    \begin{equation*}
        \Sigma = \left\lbrace \sigma : [0, T] \to \mathcal{E} \;\vert\; (\pi^{\mathcal{E}} \circ \sigma)(0) \in \mathcal{M}_0, \, (\pi^{\mathcal{E}} \circ \sigma)(T) \in \mathcal{M}_T\right\rbrace\,,
    \end{equation*}
    the space of admissible curves. Then, a (Bolza-type) optimal control problem in $\Sigma$ for such a controlled system can be defined by an objective function $J: \Sigma \to \mathbb{R}$,
    \begin{equation*}
        J(\sigma) = \phi(\pi^{\mathcal{E}}(\sigma)(T)) + \int_{0}^T C(\sigma) \, d t,
    \end{equation*}
    where $\phi: \mathcal{M}_T \to \mathbb{R}$ is referred to as terminal cost or Mayer term, and $C: \mathcal{E} \to \mathbb{R}$ as running cost or Lagrange term. Thus, the optimal control problem can be stated as
    \begin{equation*}
        \min_{\sigma \in \Sigma} J(\sigma), \text{ subject to } \frac{d}{dt}(\pi^{\mathcal{E}} \circ \sigma)(t) = f(\sigma(t))
    \end{equation*}
    or, using local coordinates, as the more familiar
    \begin{equation}
        \label{eq:OCP}
        \min_{(x,u) \in \Sigma} J(x,u), \text{ subject to } \dot{x}(t) = f(x(t),u(t)).\tag{OCP}
    \end{equation}

    Pontryagin's maximum or minimum principle gives us necessary conditions for $\sigma$ to be optimal. In order to formulate the theorem, we need to define the \emph{control Hamiltonian} of the problem, a function $\mathcal{H}_{\lambda_0}: T^* \mathcal{M} \oplus_\mathcal{M} \mathcal{E} \to \mathbb{R}$, with $\lambda_0 \in \mathbb{R}$, locally written as
    \begin{equation}
        \label{eq:Pontryagins_control_Hamiltonian}
        \mathcal{H}_{\lambda_0}(x,\lambda,u) = \left\langle \lambda, f(x,u) \right\rangle_{\mathcal{M}} + \lambda_0 C(x,u)\,.
    \end{equation}
    Here $(T^* \mathcal{M}, \pi_{\mathcal{M}}, \mathcal{M})$ is the cotangent bundle of $\mathcal{M}$, $\oplus_{\mathcal{M}}$ denotes the Whitney sum over $\mathcal{M}$ and $\left\langle \cdot, \cdot \right\rangle_{\mathcal{M}}: T^*\mathcal{M} \oplus_{\mathcal{M}} T \mathcal{M} \to \mathbb{R}$ is the natural pairing of covectors and vectors on $\mathcal{M}$. The covector $\lambda$ receives the name of \emph{costate} or \emph{adjoint variable}. $\lambda_0$ receives the name of \emph{abnormal multiplier}. Since, whenever $\lambda_0 \neq 0$, its effective contribution in the theorem only serves to rescale the covector $\lambda$, one can restrict to $\lambda_0 \in \{-1,0,1\}$. In the case, $\lambda_0 \in \{-1,0\}$, we can talk about Pontryagin's maximum principle and minimum with $\lambda_0 \in \{0,1\}$.

    \begin{theorem}[Pontryagin's maximum principle (PMP) \cite{Clarke13}]
    Let $x \in W^{1,\infty}([0,T],\mathcal{M})$, $u \in L^{\infty}([0,T],\mathcal{N})$ s.t. $(x,u) \in \Sigma$. Let $\phi \in C^1(\mathcal{M}_T, \mathbb{R})$, and $C \in C^0(\mathcal{E},\mathbb{R})$, $f \in C^0(\mathcal{E},T\mathcal{M})$, both continuously differentiable in $x$. Further, assume $(\bar{x},\bar{u}) \in \Sigma$ is a local minimizer of \eqref{eq:OCP} and consider another curve $(\bar{x},\lambda): [0,T] \to T^* \mathcal{M}$, and a number $\lambda_0 \in \left\lbrace -1, 0 \right\rbrace$. Then, the following conditions are satisfied:
    \begin{itemize}
        \item (non-triviality) $(\lambda_0,\lambda(t)) \neq 0$, $\forall t \in [0,T]$,
        \item (transversality) $\lambda(0) \in (T_{x(0)} \mathcal{M}_0)^0$, $\lambda(T) - \lambda_0 \phi'(x(T)) \in (T_{x(T)} \mathcal{M}_T)^0$,
        \item (adjoint dynamics) $\dot{\lambda}(t) = - D_1 \mathcal{H}_{\lambda_0}(\bar{x}(t),\lambda(t),\bar{u}(t))$, for almost all $t \in [0,T]$,
        \item (maximization) $\mathcal{H}_{\lambda_0}(\bar{x}(t),\lambda(t),\bar{u}(t)) \geq \mathcal{H}_{\lambda_0}(\bar{x}(t),\lambda(t),u(t))$, $\forall u(t) \in (\pi^{\mathcal{E}})^{-1}(\bar{x}(t))$.
    \end{itemize}
    \end{theorem}

    Here, $(T_x \mathcal{M})^{0}$ denotes the annihilator of $T_x \mathcal{M}$, defined as
    \begin{equation*}
        (T_x \mathcal{M})^{0} = \left\lbrace \lambda \in T_x^* \mathcal{M} \;\vert\; \left\langle \lambda, v\right\rangle_{\mathcal{M}} = 0, \,\forall v \in T_{x} \mathcal{M} \right\rbrace.
    \end{equation*}
    As written, the previous result applies even if control constraints are considered, i.e. $\mathcal{N}$ is not a vector space but a subset of it. If both $C \in C^1(\mathcal{E},\mathbb{R})$ and $f \in C^1(\mathcal{E},T\mathcal{M})$, then the maximization condition reduces to
    \begin{equation}
        \label{eq:smooth_maximization_cond}
        D_3 \mathcal{H}_{-1}(\bar{x},\lambda,\bar{u}) = 0
    \end{equation}
    for all $t$ in the unconstrained case.\\
    
    Whenever an optimum requires $\lambda_0 = 0$, it is called an abnormal solution. However, it can be shown that if $\mathcal{M}_0 = \left\lbrace x_0 \in \mathcal{M} \right\rbrace$ and $\mathcal{M}_T \in \mathrm{int}\, \mathcal{A}(T,x_0)$ with respect to the topology of $\mathcal{M}$, or $\mathcal{M}_T = \mathcal{M}$, then $\lambda_0$ can be set to $-1$ \cite{Clarke13}. We will restrict to this case. 

    Clearly, in a thorough study of optimal control problems, the analytic regularity, i.e. the smoothness class, of these curves and functions is very important, particularly whenever $C$ or $f$ are not continuously differentiable with respect to $u$ or constraints on controls are to be imposed, i.e. if only a submanifold with boundary of $\mathcal{E}$ is considered \cite{Sussmann98,Liberzon12,Clarke13}.
    
    Our intention here, however, is to restrict ourselves to the unconstrained case and assume all functions are sufficient differentiable. In particular, it is sufficient to assume $C \in C^1(\mathcal{E},\mathbb{R})$, $f \in C^1(\mathcal{E},T\mathcal{M})$, then $x \in C^2([0,1],\mathcal{M})$ and $u \in C^1([0,T],\mathcal{N})$, and the optimal control can be tackled with the standard tools of the calculus of variations \cite{Barbero07}. However, in the following we will assume for simplicity that all functions are $C^{\infty}$. In this case, we can consider the extremization of the functional defined by the augmented objective function
    \begin{equation}
        \label{eq:generic_augmented_objective}
        \tilde{J}(x,\lambda,u) = \phi(x(T)) + \int_{0}^T \left[ C(x(t),u(t)) + \left\langle \lambda(t), \dot{x}(t) - f(x(t),u(t))\right\rangle_{\mathcal{M}} \right]  \, d t.
    \end{equation}
    Here, $(x,\lambda,u)$ is the local trivialization of a curve on $T^*\mathcal{M} \oplus_{\mathcal{M}} \mathcal{E}$. The resulting necessary optimality conditions read
    \begin{itemize}
        \item (state dynamics) $\dot{x}(t) = f(x(t),u(t))$,
        \item (adjoint dynamics) $\dot{\lambda}(t) = D_1 C(x(t),u(t)) - D_1 f(x(t),u(t))^* \lambda(t)$,
        \item (transversality) $\lambda(T) = - \phi'(x(T))$,
        \item (maximization) $0 = D_2 C(x(t),u(t)) - D_2 f(x(t),u(t))^* \lambda(t)$.
    \end{itemize}
    where ${\cdot}^*$ denotes the adjoint under the natural pairing.\\

\begin{remark}
    Using matrix notation, if both states and adjoints are assumed column matrices, the adjoint dynamics and the maximization condition take the form
    \begin{align*}
        \dot{\lambda}(t)^{\top} &= D_1 C(x(t),u(t)) - \lambda(t)^{\top} D_1 f(x(t),u(t))\,\\
        0 &= D_2 C(x(t),u(t)) - \lambda(t)^{\top} D_2 f(x(t),u(t))\,,
    \end{align*}
    respectively. We will use this notation later on for local computations.
\end{remark}

    The previous equations coincide with those from PMP under the previously stated conditions. In particular, we recuperate the differentiable maximization condition \eqref{eq:smooth_maximization_cond} for all $t$.

    The fact that we only have this smooth version of the condition at our disposal has some important consequences for the solvability of the OCP. Let us introduce some definitions in this regard.

\begin{definition}
    The vector bundle morphism $\mathbb{F}C : \mathcal{E} \to \mathcal{E}^*$ over the identity defined by
    \begin{equation*}
        \left\langle \mathbb{F}C(x,u), w
 \right\rangle_{\mathcal{E}} = \left.\frac{d}{dt} C(x,u+tw) \right\vert_{t=0}\,,
    \end{equation*}
    is called the fiber derivative of $C$. Locally $\mathbb{F}C(x,u) = (x, D_2 C(x,u))$.
\end{definition}

\begin{definition}
    If \eqref{eq:smooth_maximization_cond} establishes a local bundle map from $T^*\mathcal{M}$ to $\mathcal{E}$ over the identity, we say the OCP posed by \eqref{eq:generic_augmented_objective} is algebraically regular. Otherwise, we say it is algebraically singular.
    Further, if the fiber derivative is a local diffeomorphism we say algebraically superregular. If it establishes a global diffeomorphism, then we say it is algebraically hyperregular.
\end{definition}

These definitions formalize the fact that in some cases \eqref{eq:smooth_maximization_cond} allows us to find a relationship between $u$ and $(x,\lambda)$. If the OCP is algebraically singular, this equation fails to provide the necessary relationship and one needs to resort to the more general inequality condition. If the OCP is algebraically regular, then the equation establishes this relationship implicitly. Locally, this means that the matrix $D_3 D_3 \mathcal{H}_{-1}(x,\lambda,u)$ is of full rank \cite{Barbero07}. If the OCP is algebraically super/hyperregular, then we may establish a locally/globally isomorphic relation between $u$ and all or a subset of $\lambda$. This nomenclature was chosen in analogy to well-established ones for Lagrangian mechanics, which will be discussed in Section~\ref{sssec:force_controlled EL_eqs}.

\begin{example}
Let $\mathcal{M} = \mathcal{N}= \mathbb{R}$.
    \begin{enumerate}
        \item Let $C(x,u) = k(x)+h(x) u $ and $f(x,u) = f_1(x)u$. Equation \eqref{eq:smooth_maximization_cond} reduces to
    \begin{equation*}
        - \lambda f_1(x) + h(x) = 0
    \end{equation*}
    which gives us no information on $u$. Thus, the problem is algebraically singular.
   %
     %
        \item Let $C(x,u) = h(x) u$ and $f(x,u) = f_1(x)u + f_2(x) u^2$. Equation \eqref{eq:smooth_maximization_cond} reduces to
    \begin{equation*}
        -\lambda [f_1(x) + 2 f_2(x) u] + h(x) = 0
    \end{equation*}
    from which we get that $u = \frac{h(x) - \lambda f_1(x)}{2 \lambda f_2(x)}$, which means that the OCP is algebraically regular away from $f_2^{-1}(0)$ and $\lambda = 0$. However, this is clearly not algebraically superregular since $D_{22}C(x,u) = 0$ everywhere.
        \item Let $C(x,u) = g(x) u^2$ and $f(x,u) = f_1(x)u$. Equation \eqref{eq:smooth_maximization_cond} reduces to
    \begin{equation*}
        - \lambda f_1(x) + 2 g(x)u = 0
    \end{equation*}
    The problem is algebraically superregular away from $g^{-1}(0)$. Moreover, if $g$ is monotonous and never zero, the problem is algebraically hyperregular.
    \end{enumerate}
\end{example}

\begin{example} Let us consider a linear-quadratic (LQ) problem. Let $\mathcal{M} = \mathbb{R}^m$, $\mathcal{N} = \mathbb{R}^n$, $C(x,u) = u^\top R u$ with $R \geq 0$ a symmetric degenerate matrix and $f(x,u) = Ax + Bu$ with $(A,B)$ a controllable pair of matrices. Equation \eqref{eq:smooth_maximization_cond} reduces to
    \begin{equation*}
        2 R u = B^\top \lambda.
    \end{equation*} which does not permit us to express $u$ as a function of $(x, \lambda)$.
\end{example}


\subsection{Controlled second-order systems}

    We will restrict our scope further, to the case of controlled second-order systems. For this, we assume that $\mathcal{M} = T\mathcal{Q}$, with $\mathcal{Q}$ a smooth manifold. Since we will be focusing on systems derived from Lagrangian mechanics, we denote $\mathcal{Q}$ as configuration space. Locally, we write $x = (q,v)$.

    \begin{definition}[Semi-spray / SODE]
        Consider $\tau_{T\mathcal{Q}}: TT\mathcal{Q} \to T\mathcal{Q}$ and $T \tau_{\mathcal{Q}}: TT\mathcal{Q} \to T\mathcal{Q}$. Locally, if $(q,v,X_q,X_v) \in TT\mathcal{Q}$, then $\tau_{T\mathcal{Q}}(q,v,X_q,X_v) = (q,v)$ and $T\tau_{\mathcal{Q}}(q,v,X_q,X_v) = (q,X_q)$.
        
        $X \in \Gamma_{\mathrm{loc}}(\mathcal{Q},TT\mathcal{Q})$ is called a semi-spray or second-order differential equation (SODE), if $X$ is a mutual section of $(TT\mathcal{Q}, \tau_\mathcal{Q} \circ \tau_{T\mathcal{Q}},\mathcal{Q})$ and $(TT\mathcal{Q}, \tau_\mathcal{Q} \circ T\tau_{\mathcal{Q}},\mathcal{Q})$.
     
        Locally, this implies that $X_q(q,v) = v$, and $X$ defines the differential equation
        \begin{align*}
            \dot{q}(t) &= v(t),\\
            \dot{v}(t) &= X_v(q(t),v(t)),
        \end{align*}
        or, equivalently,
        \begin{align*}
            \ddot{q}(t) = X_v(q(t),\dot{q}(t)).
        \end{align*}
    \end{definition}

    In the case of a controlled system, we work on $(\mathcal{E}, \pi^{\mathcal{E}}, T\mathcal{Q})$ but focus on the resulting dynamics on $T\mathcal{Q}$, so it makes sense to consider sections of
    $TT\mathcal{Q} \oplus_{T\mathcal{Q}} \mathcal{E}$. In this space, we have the following structural projections induced by $\tau_{T\mathcal{Q}}$ and $T \tau_{\mathcal{Q}}$:
\begin{center}
    \begin{tikzcd}
        TT\mathcal{Q} \oplus_{T\mathcal{Q}} \mathcal{E} \arrow[rr, "\mathrm{pr}_2"] \arrow[d, "\mathrm{pr}_1"'] &  & \mathcal{E} \arrow[d, "\pi^{\mathcal{E}}"] \\
        TT\mathcal{Q} \arrow[rr, "\tau_{T\mathcal{Q}}"'] &  & T\mathcal{Q}
    \end{tikzcd}
    \hspace{2cm}
    \begin{tikzcd}
        TT\mathcal{Q} \oplus_{T\mathcal{Q}} \mathcal{E} \arrow[rr, "\widetilde{\mathrm{pr}}_2"] \arrow[d, "\mathrm{pr}_1"'] &  & \mathcal{E} \arrow[d, "\pi^{\mathcal{E}}"] \\
        TT\mathcal{Q} \arrow[rr, "T \tau_\mathcal{Q}"']  &  & T\mathcal{Q}
    \end{tikzcd}
\end{center}
    Locally,
    \begin{align*}
        \mathrm{pr}_1(q,v,X_q,X_v,u) &= (q,v,X_q,X_v),\\
        \mathrm{pr}_2(q,v,X_q,X_v,u) &= (q,v,u),\\
        \widetilde{\mathrm{pr}}_2(q,v,X_q,X_v,u) &= (q,X_v,u).
    \end{align*}
    
    With these we can provide the following definition.

    \begin{definition}[Controlled SODE]
        $X \in \Gamma_{\mathrm{loc}}(\mathcal{E},TT\mathcal{Q}\oplus_{T\mathcal{Q}} \mathcal{E})$ is said to be a controlled SODE if and only if it is simultaneously a section of $\mathrm{pr}_2$ and $\widetilde{\mathrm{pr}}_2$. In bundle coordinates, a section of the former has the form
        \begin{equation*}
            (q,v,X_q(q,v,u),X_v(q,v,u),u)\,.
        \end{equation*}
        Thus, being a controlled SODE implies that $X_q(q,v,u) = v$ and it defines the differential equations
        \begin{align*}
            \dot{q}(t) &= v(t),\\
            \dot{v}(t) &= X_v(q(t),v(t),u(t)),
        \end{align*}
        or, equivalently,
        \begin{align*}
            \ddot{q}(t) = X_v(q(t),\dot{q}(t),u(t)).
        \end{align*}
    \end{definition}

    \begin{remark}
        Restricting to the case of controlled SODEs, implies that we are forfeiting the possibility of having direct control over the velocities of our states.
    \end{remark}
    
    \begin{remark}
        \label{rmk:controlled_SODE_lower_control_space}
        Notice that, in contrast to the ordinary SODE case, a controlled SODE is not defined as $X \in \Gamma_{\mathrm{loc}}(\mathcal{Q},TT\mathcal{Q}\oplus_{T\mathcal{Q}} \mathcal{E})$ since this would only allow for \emph{so-called} feedback controls. Moreover, it is not generally possible to define a subbundle $\check{\mathcal{E}}$ that locally trivializes into $\mathcal{Q} \times \mathcal{N}$, i.e. something of the form $(q,u)$, without extra assumptions. However, in applications, that is frequently the case, since it is common for $\mathcal{E}$ to be itself trivial. In such a case, it makes sense to talk about sections $\Gamma_{\mathrm{loc}}(\check{\mathcal{E}},TT\mathcal{Q}\oplus_{T\mathcal{Q}} \mathcal{E})$.
    \end{remark}

    \begin{definition}[Full actuation]
        A controlled SODE $X$ is said to be fully actuated if it maps surjectively onto $T^{(2)}\mathcal{Q}$. In coordinates, this means that $X_v(q,v,u)$ is surjective.
    \end{definition}
    
    Here, $T^{(2)}\mathcal{Q}$ denotes the second-order tangent bundle \cite{deLeonRodrigues85}, a fiber bundle which can be regarded as the space of $2$-jets with source at the origin of $\mathbb{R}$ and target $\mathcal{Q}$, or the subbundle of the double tangent bundle $TT\mathcal{Q}$ defined by $\tau_{T\mathcal{Q}} = T\tau_{\mathcal{Q}}$.\\

    Whenever we have full actuation and $\dim \mathcal{N} > \dim \mathcal{Q}$, we talk about \emph{over-actuation}. For our purposes, we will always assume that $\dim \mathcal{N} \leq \dim \mathcal{Q}$ and not pursue this case any further.\\
    
    Further, we assume for simplicity that $X$ maps injectively onto $T^{(2)}\mathcal{Q}$. Thus, in the fully actuated case, i.e. $\dim \mathcal{N} = \dim \mathcal{Q}$, we have a bijection. Whenever $\dim \mathcal{N} < \dim \mathcal{Q}$ we talk about \emph{under-actuation}.

    Locally, if:
    \begin{itemize}
        \item $\mathrm{rank} \, D_3 X_v(q,v,u) = \dim \mathcal{Q}$, we have full actuation;
        \item $\mathrm{rank} \, D_3 X_v(q,v,u) < \dim \mathcal{Q}$, we have under-actuation.
    \end{itemize}

    \begin{example}
        \label{exp:actuation}
        Let $\mathcal{Q} = \mathbb{R}^d$, $\mathcal{N} = \mathbb{R}^n$ and consider $X_v(q,v,u) = f_0(q,v) + f_1(q,v) u$ with $f_1(q,v) \in M_{d,n}(\mathbb{R})$. Clearly if $n \neq d$, we cannot possibly have full actuation. With our injectivity assumption, by the rank-nullity theorem, we have full actuation if and only if $n = d$.
    \end{example}

    \subsubsection{Force-controlled Euler-Lagrange equations}
    \label{sssec:force_controlled EL_eqs}

    A Lagrangian system is defined by a pair of configuration space and Lagrangian function, $(\mathcal{Q},L)$. We assume that $L \in C^{k}(T\mathcal{Q},\mathbb{R})$, $k \geq 2$. Consider the space of functions,
    \begin{equation*}
        C^{k}([t_a,t_b],q_a,q_b) = \left\lbrace  q \in C^k([t_a,t_b],\mathcal{Q}) \; \vert \; q(t_a) = q_a, q(t_b) = q_b \right\rbrace\,.
    \end{equation*}
    The action is defined as the function
    $S : C^{k}([t_a,t_b],q_a,q_b) \to \mathbb{R}$,
    \begin{equation}
        \label{eq:Lagrangian_action}
        S(q) = \int_{t_a}^{t_b} L(q(t),\dot{q}(t)) \, dt.
    \end{equation}
    Hamilton's principle states that physical trajectories of the system are in one-to-one correspondence with critical points of the action. The equation that characterizes these critical points in adapted coordinates is the celebrated Euler-Lagrange equation,
    \begin{equation*}
        \frac{d}{dt}\left(D_2 L(q(t),\dot{q}(t))\right) - D_1 L(q(t),\dot{q}(t)) = 0\,.
    \end{equation*}
    We can expand this equation to make its second-order character more explicit,
    \begin{equation*}
       D_{2 2} L(q(t),\dot{q}(t)) \, \ddot{q}(t) + D_{1 2} L(q(t),\dot{q}(t)) \, \dot{q}(t) - D_1 L(q(t),\dot{q}(t)) = 0\,.
    \end{equation*}
    Whenever the matrix $D_{2 2} L(q,v)$ is regular, the Euler-Lagrange equation can be transformed into a SODE and the Lagrangian is said to be \emph{regular}.
    
    The map $\mathbb{F}L : T\mathcal{Q} \to T^*\mathcal{Q}$, locally defined by $(q,v) \mapsto (q, p = D_2 L(q,v))$, is called the fiber derivative and the variables $p$ receive the name of canonical momenta. Regularity implies that $\mathbb{F}L$ defines a local diffeomorphism. Whenever this can be extended to a global diffeomorphism, the Lagrangian is said to be \emph{hyperregular}.\\
    
    The cotangent bundle is a prototypical symplectic manifold, with symplectic form $\omega$.
    Let $(q^1,...,q^{\dim \mathcal{Q}},$ $p^1,...,p^{\dim \mathcal{Q}})$ define local Darboux coordinates, then
    \begin{equation*}
        \omega = dq^i \wedge dp_i\,.
    \end{equation*}
    Using the fiber derivative one can pullback the symplectic form to $T\mathcal{Q}$, generating what is known as the Poincar{\'e}-Cartan 2-form\footnote{This can be also constructed using only $L$ and the canonical machinery of the tangent bundle}, which provides a symplectic structure if $L$ is hyperregular. Locally,
    \begin{equation*}
        \omega_L = dq^i \wedge \left(\frac{\partial L}{\partial v^i}\right)\,.
    \end{equation*}

    With the aid of the Liouville field, $\triangle$, canonical in any tangent bundle, one can also define the Lagrangian energy
    \begin{equation*}
        E_{L} = \triangle L - L\,.
    \end{equation*}
    Locally, $\triangle = v^i \partial_{v^i}$, and
    \begin{equation*}
        E_{L}(q,v) = D_2 L(q,v_{q}) \, v - L(q,v)\,.
    \end{equation*}
    With it, the Euler-Lagrange equations can be rewritten in the geometric form
    \begin{equation}
        \label{eq:EL_geometric}
        \imath_{X_L} \omega_L = dE_L
    \end{equation}
    where $X_L$ is the Euler-Lagrange vector field and $\imath$ denotes the interior product of forms with vector fields. With this, one can prove that the Euler-Lagrange equation preserves the symplectic form, making its flow a symplectomorphism.

    Hamilton's principle can be generalized with the D'Alembert principle, which allows for the inclusion of external non-potential forces $f_L: T\mathcal{Q} \to T^*\mathcal{Q}$, such that $\tau_{\mathcal{Q}} = \pi_\mathcal{Q} \circ f_L$, leading to forced Euler-Lagrange equations of the form
    \begin{equation*}
        \frac{d}{dt}\left(D_2 L(q(t),\dot{q}(t))\right) - D_1 L(q(t),\dot{q}(t)) = f_L(q(t),\dot{q}(t))\,.
    \end{equation*}
    However, in this case, the symplecticity of the flow is lost.\\

    The inclusion of controls can be done at different levels. One can have controlled Lagrangians $L^\mathcal{E}: \mathcal{E} \to \mathbb{R}$, i.e. Lagrangians dependent on the controls. Instead, we will be concerned with force-controlled Lagrangian systems, where controls appear only inside forcing terms, i.e. $f_L^{\mathcal{E}}: \mathcal{E} \to T^*\mathcal{Q}$, such that $\tau_\mathcal{Q} \circ \pi^{\mathcal{E}} = \pi_\mathcal{Q} \circ f_L^\mathcal{E}$. Technically, these forces can be either potential or non-potential, but we will not make a distinction. Thus, a force-controlled Lagrangian system is defined by the quadruple $(\mathcal{Q},\mathcal{E},L,f_L^{\mathcal{E}})$, and its equations are of the form
    \begin{equation}
        \label{eq:force_controlled_EL}
        \frac{d}{dt}\left(D_2 L(q(t),\dot{q}(t))\right) - D_1 L(q(t),\dot{q}(t)) = f_L^{\mathcal{E}}(q(t),\dot{q}(t),u(t))\,.
    \end{equation}
    Clearly, if the Lagrangian is regular, then these equations can be rewritten explicitly as a controlled SODE, adopting the form
    \begin{equation*}
        \ddot{q}(t) = X_v(q(t),\dot{q}(t),u(t)),
    \end{equation*}
    with $X_v = (D_{2 2}L)^{-1} (D_1 L + f_L^{\mathcal{E}} - D_{1 2} L \, \dot{q})$.\\

    The case of controlled Lagrangians can be treated similarly. However, one must be aware of the fact that the regularity of the Lagrangian may depend upon the control, and furthermore, the resulting controlled SODE may potentially depend not only on $u$ but its first derivative.

    \begin{remark}
    \label{rmk:Lagrangian_SODE_num}
        The fact that Eq.~\eqref{eq:force_controlled_EL} may be rewritten as a controlled SODE does not mean that it is always advisable to do so. In fact, doing so can have numerical repercussions when integrating these equations. This will be revisited once more in Section~\ref{ssec:OptimalControlLagrangian}.
    \end{remark}


\section{Optimal control of second-order systems}
\label{sec:ocp_sodes}

The optimal control problem for second-order systems, under sufficient differentiability conditions can be treated exactly as in Eq.~\eqref{eq:generic_augmented_objective}. There, the augmented running cost featured inside the integral was a function on $T^*\mathcal{M} \oplus_{\mathcal{M}} \mathcal{E}$. In the second-order case, this naturally leads to an augmented running cost on $T^* T\mathcal{Q} \oplus_{T \mathcal{Q}} \mathcal{E}$. Assuming adapted local coordinates $(q,v,\lambda_q,\lambda_v)$ on $T^* T \mathcal{Q}$, 
and a generic controlled SODE $X = v \, \partial_q + f(q,v,u) \, \partial_v$, Eq.~\eqref{eq:generic_augmented_objective} transforms into
    \begin{align}
        \tilde{J}_1(x,\lambda,u) &= \phi(q(T),v(T)) \label{eq:functional_J1}\\
        &+ \int_{0}^T \left[ C(q(t),v(t),u(t)) + \left\langle (\lambda_q(t),\lambda_v(t)), (\dot{q}(t) - v(t),\dot{v}(t) - f(q(t),v(t),u(t))) \right\rangle_{T\mathcal{Q}} \right] dt\,.\nonumber
    \end{align}
We refer to this as the \emph{first-order version} of the optimal control problem since the controlled SODE appears as a first order system.

The second-order constraint, i.e. $\dot{q} = v$, may be added implicitly, leading to a new augmented cost function
    \begin{equation}
            \tilde{J}_2(y,u) = \phi(q(T),\dot{q}(T)) + \int_{0}^T \left[ C(q(t),\dot{q}(t),u(t)) + \kappa(t)^{\top} (\ddot{q}(t) - f(q(t),\dot{q}(t),u(t)))\right] dt\,. \label{eq:functional_J2}
    \end{equation}
Here, the curve $y = (q,\kappa)$ is a curve on $T^* \mathcal{Q}$. One must also understand $u$ to denote the curve defined on the fibers along the tangent lift of $q$, i.e. $(q(t),\dot{q}(t),u(t)) \in \mathcal{E}$. The remaining constraint is now to be interpreted as a function on $T^{(2)} \mathcal{Q} \oplus_{T\mathcal{Q}} \mathcal{E}$.

Taking variations we find that the necessary conditions for optimality provided by $\tilde{J}_1$ are
    \begin{itemize}
        \item (state dynamics) $\dot{q}(t) = v(t)$,\\
        \hphantom{(state dynamics)} $\dot{v}(t) = f(x(t),v(t),u(t))$,
        \item (adjoint dynamics) $\dot{\lambda}_q(t)^{\top} = D_1 C(q(t),v(t),u(t)) - \lambda_v(t)^{\top}\, D_1 f(q(t),v(t),u(t))$,\\
        \hphantom{(adjoint dynamics)} $\dot{\lambda}_v(t)^{\top} = D_2 C(q(t),v(t),u(t)) - \lambda_v(t)^{\top} \, D_2 f(q(t),v(t),u(t)) - \lambda_q(t)^{\top}$,
        \item (maximization) $0 = D_3 C(q(t),v(t),u(t)) - \lambda_v(t)^{\top} \, D_3 f(q(t),v(t),u(t))$,
        \item (transversality) $\lambda_q(T)^{\top} = -D_1 \phi(q(T),v(T))$,\\
	\hphantom{(transversality)} $\lambda_v(T)^{\top} = -D_2 \phi(q(T),v(T))$;
    \end{itemize}
while those provided by $\tilde{J}_2$ are
    \begin{itemize}
        \item (state dynamics) $\ddot{q}(t) = f(q(t),\dot{q}(t),u(t))$,
        \item (adjoint dynamics) $\ddot{\kappa}(t)^{\top} = \frac{d}{dt} \left[ D_2 C(q(t),\dot{q}(t),u(t)) - \kappa(t)^{\top} D_2 f(q(t),\dot{q}(t),u(t)) \right]$\\
        \hphantom{(adjoint dynamics) $\ddot{\kappa}(t)^{\top}$} $ -~ D_1 C(q(t),\dot{q}(t),u(t)) + \kappa(t)^{\top} \, D_1 f(q(t),\dot{q}(t),u(t))$,
        \item (maximization) $0 = D_3 C(q(t),\dot{q}(t),u(t)) - \kappa(t)^{\top} \, D_3 f(q(t),\dot{q}(t),u(t))$,
        \item (transversality) $\kappa(T)^{\top} = -D_2 \phi(q(T),\dot{q}(T))$,\\
	\hphantom{(transversality)} $\dot{\kappa}(T)^{\top} = D_1 \phi(q(T),\dot{q}(T))$\\
    \hphantom{(transversality) $\dot{\kappa}(T)^{\top}$} $+~D_2 C(q(T),\dot{q}(T),u(T)) + D_2 \phi(q(T),\dot{q}(T)) \, D_2 f(q(T),\dot{q}(T),u(T))$.
    \end{itemize}

These, lead to the following easy to check
\begin{theorem}
    The necessary optimality conditions provided by $\tilde{J}_1$ and $\tilde{J}_2$ are equivalent under the identification $\lambda_v = \kappa$.
\end{theorem}

Under the previous identification one also gets that $\lambda_q^{\top} = D_2 C(q,\dot{q},u) - \kappa^{\top} \, D_2 f(q,\dot{q},u) - \dot{\kappa}^{\top}$.\\

Consider again the augmented cost function $\tilde{J}_2$. 
Applying integration by parts once, we can remove the second derivative by transferring one differentiation to the adjoint variable $\kappa$. This results in the final version of the augmented cost function that we will consider:
    \begin{align}
		\tilde{J}_3(y,u) &= \phi(q(T),\dot{q}(T)) + \kappa(T)^{\top} \dot{q}(T) - \kappa(0)^{\top} \dot{q}(0) \label{eq:functional_J3}\\
		&+ \int_{0}^T \left[ C(q(t),\dot{q}(t),u(t)) - \dot{\kappa}(t)^{\top} \dot{q}(t) - \kappa(t)^{\top} f(q(t),\dot{q}(t),u(t)))\right] dt\,.\nonumber
    \end{align}

From the commutation of integration by parts and differentiation / variation under our smoothness assumptions, we get the following trivial
\begin{theorem}
    \label{thm:equivalence_J2_J3}
    The necessary optimality conditions provided by $\tilde{J}_2$ and $\tilde{J}_3$ are equivalent.
\end{theorem}

    Remarkably, this new augmented cost function affords us the following

    \begin{definition}
    \label{def:new_control_Lagrangian}
        Let $X$ be a controlled SODE and $C: \mathcal{E} \to \mathbb{R}$ a running cost function. Then, $\tilde{\mathcal{L}}^\mathcal{E}: TT^*\mathcal{Q} \oplus_{T\mathcal{Q}}^{\alpha} \mathcal{E} \to \mathbb{R}$,
        \begin{equation*}
            \tilde{\mathcal{L}}^\mathcal{E}(q,\kappa,v_q,v_{\kappa},u) = v_{\kappa}^{\top} v_q + \kappa^{\top} X_v(q,v_q,u) - C(q,v_q,u)\,.
        \end{equation*}
        is called a new control Lagrangian for the controlled SODE (also control-dependent new control Lagrangian).
    \end{definition}

\begin{example}\label{exp:vec_example}\hphantom{.}
        \begin{enumerate}
            \item \label{itm:vec_example_1} Let $\mathcal{Q} = \mathbb{R}$, $\mathcal{N} = \mathbb{R}$, $C(q,v_q,u) = k(q,v_q)$ and $X_v(q,v_q,u) = f_0(q,v_q) + f_3(q,v_q) u^3$, with $f_3(q,v_q) \neq 0$. The associated new control Lagrangian is
            \begin{equation*}
                \tilde{\mathcal{L}}^{\mathcal{E}}(q,\kappa,v_q,v_{\kappa},u) = v_{\kappa} v_q + \kappa [f_0(q,v_q) + f_3(q,v_q)u^3] - k(q,v_q)
            \end{equation*}
            \item \label{itm:vec_example_2} Let $(\mathcal{Q},\mathrm{g}_{\mathcal{Q}})$ Riemannian manifold. Let $(\mathcal{E},\pi^{\mathcal{E}},T\mathcal{Q},\rho)$  anchored vector bundle \cite{PopescuAnchored} and linear anchor $\rho: \mathcal{E} \to TT\mathcal{Q}$. Let $\mathrm{g}_{T\mathcal{Q}}$ be the Sasaki metric induced by $\mathrm{g}_{\mathcal{Q}}$ on $T\mathcal{Q}$ and $\mathrm{g}_{\mathcal{E}} = \rho^* \mathrm{g}_{T\mathcal{Q}}$. Finally, let $C(q,v_q,u) = \frac{1}{2} \mathrm{g}_{\mathcal{E}}(q,v_q)(u,u)$ and $X_v(q,v_q,u) = f_0(q,v_q) + f_1(q,v_q)(u)$, where the linear part in $u$ is induced by $((T\tau_{\mathcal{Q}} \circ \rho)(q,v_q,u))^v$ with $\cdot^v$ the vertical lift of vector fields. Then, the new Lagrangian is
            \begin{equation*}
                \tilde{\mathcal{L}}^{\mathcal{E}}(q,\kappa,v_q,v_{\kappa},u) = v_{\kappa}^{\top} v_q + \kappa^{\top} [f_0(q,v_q) + f_1(q,v_q)(u)] - \frac{1}{2}\mathrm{g}_{\mathcal{E}}(q,v_q)(u,u).
            \end{equation*}
        \end{enumerate}
    \end{example}

    Here, we have defined this new Lagrangian as minus the running cost in $\tilde{J}_3$ for later convenience. We leave the formal definition of $TT^*\mathcal{Q} \oplus_{T\mathcal{Q}}^{\alpha} \mathcal{E}$ for Section \ref{sec:tulcyjew_and_PMP} and provide the following abridged diagram that shows in which sense this is to be thought of as a sum of vector bundles.

    \begin{center}
        \begin{tikzcd}
        TT^*\mathcal{Q} \oplus_{T\mathcal{Q}}^{\alpha} \mathcal{E} \arrow[rr, "\mathrm{pr}_2^{\alpha}"] \arrow[dd, "\mathrm{pr}_1^{\alpha}"'] &  & \mathcal{E} \arrow[dd, "\pi^{\mathcal{E}}"] \\
                                                                                                                                              &  &                                             \\
        TT^*\mathcal{Q} \arrow[rr, "T\pi_{\mathcal{Q}}"']                                                                                     &  & T\mathcal{Q}                               
        \end{tikzcd}
    \end{center}

    This new control Lagrangian can be regarded as a $u$-parametrized family of Lagrangians on $TT^*\mathcal{Q}$. One of its most interesting features can be summarized in the following
    \begin{proposition}
        \label{prp:hyperregularity_new_Lagrangian}
        When $u$ is regarded as a parameter, $\tilde{\mathcal{L}}^{\mathcal{E}}$ is hyperregular for all $u$.
    \end{proposition}

    \begin{proof}
        To show this, let us first invoke the identity
        \begin{equation*}
            \det \left(
        \begin{array}{cc}
             \mathbf{A} & \mathbf{B} \\
             \mathbf{C} & \mathbf{D}
        \end{array}
        \right) = \det (\mathbf{A} \mathbf{D} - \mathbf{B} \mathbf{C}),
        \end{equation*}
        whenever $\mathbf{A}$, $\mathbf{B}$, $\mathbf{C}$ and $\mathbf{D}$ are square matrices and $\mathbf{C}$ and $\mathbf{D}$ commute \cite{Silvester00}. With this, it is trivial to check that the determinant of the matrix
    \begin{equation*}
        \left(
        \begin{array}{cc}
             D_{3 3} \tilde{\mathcal{L}}^{\mathcal{E}} & D_{4 3} \tilde{\mathcal{L}}^{\mathcal{E}} \\
             D_{3 4} \tilde{\mathcal{L}}^{\mathcal{E}} & D_{4 4} \tilde{\mathcal{L}}^{\mathcal{E}} 
        \end{array}
        \right) =  \left(
        \begin{array}{cc}
             D_{3 3} \tilde{\mathcal{L}}^{\mathcal{E}} & I_{\dim \mathcal{Q}} \\
             I_{\dim \mathcal{Q}} & 0_{\dim \mathcal{Q}} 
        \end{array}
        \right)
    \end{equation*}
    is point-independent and of value $(-1)^{\dim \mathcal{Q}-1}$.
    This shows that the associated fiber derivative is a diffeomorphism.
    \end{proof}

    \begin{remark}
        Do note that the former result is completely independent of the actuation type of the problem. This already hints  towards a certain canonicity of these transformations, as we will see in Section~\ref{sec:tulcyjew_and_PMP}.
    \end{remark}


    One can also regard $\tilde{\mathcal{L}}^{\mathcal{E}}$ locally as a singular Lagrangian on $T(T^*\mathcal{Q} \times \mathcal{N})$, with $u$ playing the role of a Lagrange multiplier of sorts enforcing the maximization condition. Owed to this, it is trivial to check that the necessary optimality conditions for $\tilde{J}_3$ are given by
        \begin{itemize}
            \item (adjoint dynamics) $0 = D_1 \tilde{\mathcal{L}}^\mathcal{E}(q(t),\kappa(t),\dot{q}(t),\dot{\kappa}(t),u(t)) - \frac{d}{dt}\left( D_3 \tilde{\mathcal{L}}^\mathcal{E}(q(t),\kappa(t),\dot{q}(t),\dot{\kappa}(t),u(t)) \right)$
            \item (state dynamics) $0 = D_2 \tilde{\mathcal{L}}^\mathcal{E}(q(t),\kappa(t),\dot{q}(t),\dot{\kappa}(t),u(t)) - \frac{d}{dt}\left( D_4 \tilde{\mathcal{L}}^\mathcal{E}(q(t),\kappa(t),\dot{q}(t),\dot{\kappa}(t),u(t)) \right)$,
        \item (maximization) $0 = D_5 \tilde{\mathcal{L}}^\mathcal{E}(q(t),\kappa(t),\dot{q}(t),\dot{\kappa}(t),u(t))$\,,
    \end{itemize}
together with the transversality conditions for $\tilde{J}_2$.

\begin{definition}
    Assume $\bar{u}$ is such that the maximization condition is satisfied. Then,
    $\tilde{\mathcal{L}}(q,\kappa,v_q,v_{\kappa}) := \tilde{\mathcal{L}}^{\mathcal{E}}(q,\kappa,v_q,v_{\kappa},\bar{u})$ is called a (locally) optimal new control Lagrangian for the controlled SODE (also control-independent new control Lagrangian).
\end{definition}

Since by assumption $D_5 \tilde{\mathcal{L}}^{\mathcal{E}}(q,\kappa,v_q,v_{\kappa},\bar{u}) = 0$ we have that
\begin{align*}
    D_i \tilde{\mathcal{L}}(q,\kappa,v_q,v_{\kappa}) &= D_i \tilde{\mathcal{L}}^{\mathcal{E}}(q,\kappa,v_q,v_{\kappa},\bar{u}) + D_5 \tilde{\mathcal{L}}^{\mathcal{E}}(q,\kappa,v_q,v_{\kappa},\bar{u}) D_i \bar{u}(q,\kappa,v_q,v_{\kappa})\\
    &= D_i \tilde{\mathcal{L}}^{\mathcal{E}}(q,\kappa,v_q,v_{\kappa},\bar{u}), \quad i = 1, ..., 4.
\end{align*}
Consequently, the Euler-Lagrange equations of $\tilde{\mathcal{L}}$ are equivalent to those of $\tilde{\mathcal{L}}^{\mathcal{E}}$ after substitution of $\bar{u}$.
\begin{example}[Cont.~of Example \ref{exp:vec_example}]
    \label{exp:vec_example_cont}
    The maximization condition tells us that 
        \begin{enumerate}
            \item \label{itm:vec_example_cont_1} 
            $3 \kappa f_3(q, v_q)\bar{u}^2 = 0$, so that $\bar{u} = 0$ away from $\kappa = 0$ and
            \begin{equation*}
                \tilde{\mathcal{L}}(q,\kappa,v_q,v_{\kappa}) = v_{\kappa} v_q + \kappa f_0(q,v_q) - k(q,v_q).
            \end{equation*}
            For $\kappa = 0$ the problem leads to singular controls.
            \item \label{itm:vec_example_cont_2} 
            $\kappa^{\top} f_1(q,v_q) - \mathrm{g}_{\mathcal{E}}(q,v_q)(\bar{u},\cdot) = 0$,
            so $\bar{u} = \mathrm{g}_{\mathcal{E}}^{-1}(q,v_q)(f_1(q,v_q)^{\top}(\kappa))$ and
            \begin{equation*}
                \tilde{\mathcal{L}}(q,\kappa,v_q,v_{\kappa}) = v_{\kappa}^{\top} v_q + \kappa^{\top} f_0(q,v_q) + \frac{1}{2}\mathrm{g}_{\mathcal{E}}^{-1}(q,v_q)(f_1(q,v_q)^{\top}(\kappa),f_1(q,v_q)^{\top}(\kappa)).
            \end{equation*}
        \end{enumerate}
    \end{example}
\begin{remark}
    The setting in \cite{Leyendecker2024new} is a particular case of the one presented here. There, only running costs quadratic in the control and affine-controlled SODEs were considered. In particular, it was assumed that the configuration space of the system was a Riemannian manifold $(\mathcal{Q}, \mathrm{g}_{\mathcal{Q}})$, and that the control space was an anchored vector bundle $(\mathcal{E}, \pi^{\mathcal{E}}, \mathcal{Q},\rho)$ with injective linear anchor $\rho: \mathcal{E} \to T\mathcal{Q}$. Notice that, in contrast with our general setting, $\mathcal{E}$ was a vector bundle over $\mathcal{Q}$ and not $T\mathcal{Q}$. This meant that locally the anchor was represented by $\rho(q)$. Moreover, this structure induced a Riemannian structure on $\mathcal{E}$, namely $(\mathcal{E}, \mathrm{g} = \rho^* \mathrm{g}_{\mathcal{Q}})$. These assumptions implied that the OCP was always algebraically hyperregular.
    \begin{itemize}
        \item  Quadratic cost: $C(q,v,u) = \frac{1}{2} \mathrm{g}(q)(u,u)$.
        \item Affine-controlled SODE: $\ddot{q}^i(t) = f^i(q(t),\dot{q}(t)) + \rho^i_{\alpha}(q(t)) u^{\alpha}(t)$, $i = 1, ..., \dim \mathcal{Q}; \,\alpha = 1,..., \dim \mathcal{N}$.
    \end{itemize}
    These resulted in the new control Lagrangian
    \begin{equation*}
        \tilde{\mathcal{L}}^{\mathcal{E}}(q,\kappa,v_q,v_{\kappa},u) = v_{\kappa}^{\top} v_{q} + \kappa^{\top} [ f(q,v_q) + \rho(q)u ] - \frac{1}{2} \mathrm{g}(q)(u,u)\,.
    \end{equation*}
    The maximization condition then acquired the simple form
    \begin{equation*}
        \mathrm{g}(q)\,u = \kappa^{\top} \rho(q),
    \end{equation*}
    which, since $\mathrm{g}$ is necessarily regular, meant that $u$ could be solved for explicitly. Substituting it in $\tilde{\mathcal{L}}^{\mathcal{E}}$ led to
    \begin{equation*}
        \tilde{\mathcal{L}}(q,\kappa,v_q,v_{\kappa}) = v_{\kappa}^{\top} v_{q} + \kappa^{\top} f(q,v_q) + \frac{1}{2} b(q)(\kappa, \kappa)
    \end{equation*}
    where $b$ is a possibly degenerate quadratic form given by $\mathrm{g}$ and $\rho^{\top}$.
\end{remark}

\subsection{Higher-order Lagrangian formulation}
\label{sec:high_order_lagrangians}
An alternative to this new Lagrangian approach consists in reformulating the problem as a higher-order Lagrangian one \cite{Colombo2010,Colombo2016}. In particular, for second-order systems, the approach can be summarized as follows.

Consider a controlled SODE, which may be either fully or underactuated. By the implicit function theorem it is possible to find $u: \mathcal{W} \subset T^{(2)}\mathcal{Q} \to \mathcal{E}$, such that
\begin{equation*}
    \ddot{q} = X_v(q,\dot{q},u(q,\dot{q},\ddot{q}))\,.
\end{equation*}
In the fully actuated case $\mathcal{W} \equiv T^{(2)}\mathcal{Q}$. In the underactuated case, let us assume for simplicity that one can find coordinates in $T^{(2)}\mathcal{Q}$, $(q^b,q^{\beta},v^b,v^{\beta},a^b,a^{\beta})$ with $b = 1,..., \dim \mathcal{N}$, $\beta = \dim \mathcal{N} + 1,..., \dim \mathcal{Q}$, such that the controlled SODE reduces to
\begin{subequations}
    \label{eq:actuation_adapted_coords}
    \begin{align}
    \ddot{q}^b &= X_v^b(q,\dot{q},u)\,,\\
    \ddot{q}^{\beta} &= X_v^{\beta}(q,\dot{q})\,,
    \end{align}
\end{subequations}
with $\mathrm{rank} \, D_3 (X_v^1, ..., X_v^{\dim \mathcal{N}}) = \dim \mathcal{N}$. Then the first $\dim \mathcal{N}$ equations define $\mathcal{W}$, and we have $u(q^i,v^i,a^b)$ with $i = 1,...,\dim\mathcal{Q}$.
With this, one can define an associated second-order Lagrangian, $\mathscr{L}: T^{(2)}\mathcal{Q} \to \mathbb{R}$ \cite{deLeonRodrigues85},
\begin{equation}
    \label{eq:equiv_second_order_lagrangian}
    \mathscr{L}(q,v,a) = C(q,v,u(q,v,a))\,,
\end{equation}
to the OCP.\\

Similar to the first-order case, a second-order Lagrangian defines a fiber derivative, $\mathbb{F}\mathscr{L}: T^{(3)}\mathcal{Q} \to T^*T\mathcal{Q}$, locally defined by $(q,v,a,\jmath) \mapsto (q,v,p_q = D_2 \mathscr{L}(q,v,a) - (\frac{d}{dt} D_3 \mathscr{L})(q,v,a,
\jmath), p_v = D_3 \mathscr{L}(q,v,a))$. With this, the definition of regularity and hyperregularity of second-order Lagrangians is the same as for first-order ones. For regularity, it suffices to check that the matrix $D_{33}\mathscr{L}(q,v,a)$ is regular.

\textbf{@All}, if anyone of you could double-check this proof, it would be great.


\begin{theorem}
    \label{thm:regularity_OCP_lagrangian}
    The Lagrangian \eqref{eq:equiv_second_order_lagrangian} is regular if and only if 
    both
    \begin{enumerate}
        \item the OCP is algebraically superregular, and
        \item the controlled SODE is fully actuated and regular, i.e. its derivative is non-zero.
    \end{enumerate}
    Moreover, if the OCP is algebraically hyperregular, then the Lagrangian \eqref{eq:equiv_second_order_lagrangian} will be hyperregular.
\end{theorem}
\begin{proof}
    If the controlled SODE is fully actuated, then, under our assumptions on $X$ it establishes a bijection between $\mathcal{E}$ and $T^{(2)}\mathcal{Q}$ over $T\mathcal{Q}$. Moreover, when $X$ is regular, by the inverse function theorem, its inverse is also continuously differentiable and thus it establishes a diffeomorphism.\\\\
    If the OCP is algebraically regular, then there exist local surjective submersions from $T^*T\mathcal{Q}$ to $\mathcal{E}$ over the identity around every point.
    Further, the quotient space induced by such maps is a manifold (see \cite{MargalefOuterelo92}, Chapter 4.3.).
    In the case at hand, since only $\lambda_v$ (or equivalently $\kappa$) 
    appear in the maximization condition, this manifold is isomorphic to a submanifold of $T\mathcal{Q} \oplus_{\mathcal{Q}} T^*\mathcal{Q}$. Now, since the induced map on the quotient is both surjective and injective by construction it is a bijection and also both an immersion and a submersion. Thus, in the case of full actuation, by a dimension counting argument these map generate local diffeomorphisms from $T\mathcal{Q} \oplus_{\mathcal{Q}} T^*\mathcal{Q}$ to $\mathcal{E}$.\\\\
    With these two local diffeomorphic relations between $T\mathcal{Q} \oplus_{\mathcal{Q}} T^{*}\mathcal{Q}$ and $\mathcal{E}$ and $\mathcal{E}$ and $T^{(2)}\mathcal{Q}$, we can finally establish a local diffeomorphic relation between $T\mathcal{Q} \oplus_{\mathcal{Q}} T^{*}\mathcal{Q}$ and $T^{(2)}\mathcal{Q}$.\\\\
    Now, regularity of the OCP does not provide warranties on the properties of the running cost $C$ itself. For this, we need superregularity. Under this assumption, it is easy to show that $D_3 \mathscr{L} = \lambda_v$. Provided the functions involved are at least $C^2$, the complete fiber derivative extends to a local diffeomorphism between $T^* T\mathcal{Q}$ and $T^{(3)} \mathcal{Q}$, which is precisely the condition that $\mathscr{L}$ \eqref{eq:equiv_second_order_lagrangian} be regular. The hyperregular case follows by changing all local considerations into global ones.\\\\
    Finally, let us consider the reverse implication. If the original OCP is known, then by the uniqueness of the constructions on the direct implication, there is no other way to connect $\mathscr{L}$ with $C$ and $f$. If the OCP is unknown, a given second-order Lagrangian system 
    trivially provides us with the fully-actuated controlled SODE $X = v \partial_q + a \partial_v$. Further, $\mathscr{L}$ itself may be considered as running cost, and thus, the resulting OCP is algebraically superregular or hyperregular whenever $\mathscr{L}$ is, respectively, regular or hyperregular.
\end{proof}
\begin{remark}
    The previous result can be extended to $n$-th-order Lagrangians by providing a suitable definition of full actuation. For instance, in the first-order case, full actuation can be extended by requiring the ODE $\dot{x} = f(x,u)$ provides a bijection between $\mathcal{E}$ and $T\mathcal{Q}$. By superregularity of the OCP, we also have a local diffeomorphism between $T\mathcal{Q}$ and $\mathcal{E}$. 
\end{remark}
\begin{example}
    Let us consider a simple example. Let $C$ be quadratic in $u$ and $X_v$ be linear in $u$. In this case the maximization condition tells us that locally
    \begin{equation*}
        D_3 C(q,v,u) = \lambda_v^\top D_3 X_v(q,v,u) \Rightarrow F(q,v) u = \lambda_v^\top G(q,v).
    \end{equation*}
    If the OCP is algebraically superregular, then $F(q,v)$ is regular and $u(q,v,\lambda_v) = F(q,v)^{-1} \lambda_v^\top G(q,v)$, providing the relation from $T\mathcal{Q} \oplus T^*\mathcal{Q}$ to $\mathcal{E}$. In the case of full actuation, $G$ must be a regular square matrix and thus the previous relation is a local diffeomorphism. Now, full actuation gives us $a = X_v(q,v,u)$, and the relationship between $\mathcal{E}$ and $T^{(2)}\mathcal{Q}$ being diffeomorphic, means that we can write $u(q,v,a)$, as mentioned at the beginning of Section \ref{sec:high_order_lagrangians}.
    Now, from \eqref{eq:equiv_second_order_lagrangian}, and the definition of $a$, it is easy to check that
    \begin{equation*}
        D_3 \mathscr{L}(q,v,a) = \lambda_v^T\,.
    \end{equation*}
    From the theorem, this relationship allows us to find $a = a(q,v,\lambda_v) = X_v(q,v,u(q,v,\lambda_v))$, and thus the Lagrangian needs to be regular. In particular, $D_{3 3}L(q,v,a)$ must be regular, so
    \begin{align*}
        D_2 \mathscr{L}(q,v,a) - \frac{d}{dt} D_3 \mathscr{L}(q,v,a) &= D_2 \mathscr{L}(q,v,a) - D_{3 1} \mathscr{L}(q,v,a) v - D_{3 2} \mathscr{L}(q,v,a) a - D_{3 3} \mathscr{L}(q,v,a) \jmath\\
        &= D_2 C(q,v,u(q,v,a)) - \lambda_v^{\top} D_2 X_v(q,v,u(q,v,a)) - \dot{\lambda}_v^{\top} = \lambda_q^\top,
    \end{align*}
    which shows that $\mathbb{F}\mathscr{L}$ is a local diffeomorphism from $T^{(3)}\mathcal{Q} \to T^*T\mathcal{Q}$.
\end{example}
\begin{example}[Cont.~of Example \ref{exp:vec_example}]
    \label{exp:vec_example_high}\hphantom{.}
        \begin{enumerate}
            \item \label{itm:vec_example_high_1} $C$ does not depend on $u$, so no associated high-order Lagrangian exists.
            \item \label{itm:vec_example_high_2} Depending on the dimension of $\mathcal{N}$ and the properties of the anchor our problem will be under or fully actuated. If $\dim \mathcal{N} = \dim \mathcal{Q}$ and the anchor is injective, then $f_1(q,v\equiv v_q)$ is a regular matrix and
            \begin{equation*}
                u = (f_1 (q,v))^{-1}(a - f_0(q,v))
            \end{equation*}
            so that
            \begin{equation*}
                \mathscr{L}(q,v,a) = \frac{1}{2} \mathrm{g}_{\mathcal{E}}(q,v)((f_1 (q,v))^{-1}(a - f_0(q,v)),(f_1 (q,v))^{-1}(a - f_0(q,v)))
            \end{equation*}
            If $\dim \mathcal{N} < \dim \mathcal{Q}$, one may work as in \eqref{eq:actuation_adapted_coords}. However, the resulting Lagrangian will not be regular.
        \end{enumerate}
\end{example}


Provided the Lagrangian is at least regular, the OCP can be formulated as the Lagrangian problem
\begin{equation*}
    \mathcal{S}(q) = \phi(q(T),\dot{q}(T)) + \int_{0}^T \mathscr{L}(q(t),\dot{q}(t),\ddot{q}(t))\, dt
\end{equation*}

In any other case, the resulting Lagrangian problem will be singular, making the new Lagrangian approach more advantageous when dealing with a simply regular OCP or a super/hyperregular OCP for an underactuated controlled SODE. In the former, a second-order Lagrangian may not even exist and in the latter, the second-order Lagrangian approach requires us to work in a constrained setting. In particular, for the latter, one can define an augmented second-order Lagrangian \cite{Colombo2010}
\begin{equation*}
    \hat{\mathscr{L}}(q,v,a,\lambda) = \mathscr{L}(q,v,a) + \lambda_{\beta} (a^{\beta} - X_v^{\beta}(q,v))\,,\\
\end{equation*}
where the non-directly actuated equations have been attached as constraints. This is necessary since otherwise the system would not have information about this part of the dynamics.


\subsection{New control Hamiltonian for second-order systems}

    Proposition~\ref{prp:hyperregularity_new_Lagrangian} indicates that, since $\tilde{\mathcal{L}}^{\mathcal{E}}$ is always hyperregular, there \emph{always} exists a dual Hamiltonian formulation. We refer to this as a new control Hamiltonian for a controlled SODE. Indeed, consider the fiber derivative induced by $\tilde{\mathcal{L}}^{\mathcal{E}}$, namely, $\mathbb{F}\tilde{\mathcal{L}}^{\mathcal{E}}: TT^*\mathcal{Q} \oplus_{T\mathcal{Q}}^{\alpha} \mathcal{E} \to T^*T^*\mathcal{Q} \oplus_{T\mathcal{Q}}^{\beta} \mathcal{E}$,
    \begin{equation*}
        \mathbb{F}\tilde{\mathcal{L}}^{\mathcal{E}}( q,  \kappa, v_q, v_{\kappa}, u) = \left(q, \kappa, p_{q} = D_{3}\tilde{\mathcal{L}}^{\mathcal{E}}( q,  \kappa, v_q, v_{\kappa}, u), p_{\kappa} = D_{4}\tilde{\mathcal{L}}^{\mathcal{E}}( q,  \kappa, v_q, v_{\kappa}, u), u\right)\,.
    \end{equation*}

    More explicitly, using matrix notation
    \begin{align*}
        p_q^{\top} &= v_{\kappa}^{\top} + \kappa^{\top} D_2 X_v(q,v_q,u) - D_2 C(q,v_q,u)\,,\\
        p_{\kappa}^{\top} &= v_q^{\top}\,.
    \end{align*}

    We leave the formal definition of $T^*T^*\mathcal{Q} \oplus_{T\mathcal{Q}}^{\beta} \mathcal{E}$ for Section \ref{sec:tulcyjew_and_PMP}.\\

    We may now define the energy function associated to our new control Lagrangian,
    \begin{equation*}
        E_{\tilde{\mathcal{L}}^{\mathcal{E}}} = \triangle \tilde{\mathcal{L}}^{\mathcal{E}} - \tilde{\mathcal{L}}^{\mathcal{E}}
    \end{equation*}
    where $\triangle$ denotes the Liouville field brought from $TT^*\mathcal{Q}$ to $TT^*\mathcal{Q} \oplus_{TQ}^{\alpha} \mathcal{E}$. Locally, $\triangle = v_q^i \partial_{v_q^i} + v_{\kappa}^i \partial_{v_{\kappa}^i}$ and
    \begin{equation*}
        E_{\tilde{\mathcal{L}}^{\mathcal{E}}}(q,\kappa,v_{q},v_{\kappa},u) = D_3 \tilde{\mathcal{L}}^{\mathcal{E}}(q,\kappa,v_{q},v_{\kappa},u) \, v_q + D_4 \tilde{\mathcal{L}}^{\mathcal{E}}(q,\kappa,v_{q},v_{\kappa},u) \, v_{\kappa} - \tilde{\mathcal{L}}^{\mathcal{E}}(q,\kappa,v_{q},v_{\kappa},u)\,.
    \end{equation*}

    It is trivial to show that when restricted to the optimal new control Lagrangian, i.e. $E_{\tilde{\mathcal{L}}}(q,\kappa,v_{q},v_{\kappa}) := E_{\tilde{\mathcal{L}}^{\mathcal{E}}}(q,\kappa,v_{q},v_{\kappa},\bar{u})$ with $\bar{u}$ satisfying the maximization condition, this energy function is a conserved quantity of the flow induced by the state-adjoint dynamics.\\

    With this, we can provide the following
    \begin{definition}
        Let $\tilde{\mathcal{L}}^{\mathcal{E}}$ be a new control Lagrangian for a controlled SODE $X$ with running cost $C$. Then, the Legendre transform of this new control Lagrangian, i.e.
        $\tilde{\mathcal{H}}^{\mathcal{E}} = (E_{\tilde{\mathcal{L}}^{\mathcal{E}}} \circ \mathbb{F}\tilde{\mathcal{L}}^{\mathcal{E}}) : T^*T^*\mathcal{Q} \oplus_{T\mathcal{Q}}^{\beta} \mathcal{E} \to \mathbb{R}$,
        \begin{equation} \label{eq:new.control.Hamiltonian}
             \tilde{\mathcal{H}}^{\mathcal{E}}(q,\kappa,p_q,p_{\kappa},u) = p_{q}^{\top} p_{\kappa} - \kappa^{\top} X_v(q,p_{\kappa},u) + C(q,p_{\kappa},u),
        \end{equation}
        is called a new control Hamiltonian for the controlled SODE (also control-dependent new control Hamiltonian).
    \end{definition}

    The necessary conditions for optimality transform into Hamilton's equation once more, as with the standard control Hamiltonian. However, in contrast to those or the new Lagrangian case, state and adjoint dynamics appear mixed together
    \begin{itemize}
        \item (state-adjoint dynamics) $\dot{q}(t) = D_3 \tilde{\mathcal{H}}^{\mathcal{E}}(q(t),\kappa(t),p_q(t),p_{\kappa}(t),u(t))$,\\
        \hphantom{(state-adjoint dynamics)} $\dot{\kappa}(t) = D_4 \tilde{\mathcal{H}}^{\mathcal{E}}(q(t),\kappa(t),p_q(t),p_{\kappa}(t),u(t))$,\\
        \hphantom{(state-adjoint dynamics)} $\dot{p}_{q}(t) = -D_1 \tilde{\mathcal{H}}^{\mathcal{E}}(q(t),\kappa(t),p_q(t),p_{\kappa}(t),u(t))$,\\
        \hphantom{(state-adjoint dynamics)} $\dot{p}_{\kappa}(t) = -D_2 \tilde{\mathcal{H}}^{\mathcal{E}}(q(t),\kappa(t),p_q(t),p_{\kappa}(t),u(t))$,
        \item (maximization) $0 = D_5 \tilde{\mathcal{H}}^{\mathcal{E}}(q(t),\kappa(t),p_q(t),p_{\kappa}(t),u(t))$\,,
        \item (transversality) $\kappa(T)^{\top} = -D_2 \phi(q(T),p_{\kappa}(T))$,\\
	\hphantom{(transversality)} $p_q(T) = D_1 \phi(q(T),p_{\kappa}(T))$.
    \end{itemize}

\begin{remark}
    Using Pontryagin's control Hamiltonian with $\lambda_0 = -1$,
    \begin{equation*}
        \mathcal{H}_{-1}(q,v,\lambda_q,\lambda_v,u) =  \left\langle (\lambda_q,\lambda_v), (v, X_v(q,v,u)) \right\rangle_{T\mathcal{Q}} - C(q,v,u),
    \end{equation*}
    the state and adjoint dynamics of $\tilde{\mathcal{J}}_1$ adopt the form
    \begin{itemize}
        \item (state dynamics) $\dot{q}(t) = D_3 \mathcal{H}_{-1}(q(t),v(t),\lambda_q(t),\lambda_v(t),u(t))$,\\
        \hphantom{(state dynamics)} $\dot{v}(t) = D_4 \mathcal{H}_{-1}(q(t),v(t),\lambda_q(t),\lambda_v(t),u(t))$,
        \item (adjoint dynamics) $\dot{\lambda}_{q}(t) = -D_1 \mathcal{H}_{-1}(q(t),v(t),\lambda_q(t),\lambda_v(t),u(t))$,\\
        \hphantom{(adjoint dynamics)} $\dot{\lambda}_v(t) = -D_2 \mathcal{H}_{-1}(q(t),v(t),\lambda_q(t),\lambda_v(t),u(t))$.
    \end{itemize}
    While structurally identical to the ones from $\tilde{\mathcal{H}}^{\mathcal{E}}$, the spaces where they belong are different, and so is \emph{a priori} the presymplectic structure behind them (see Section \ref{sec:tulcyjew_and_PMP}). This particular point can have important repercussions for the discretization process of the theory and the resulting numerical methods.
\end{remark}
    

\begin{definition}
    Assume $\bar{u}(t)$ is such that the maximization condition is satisfied. Then,
    $\tilde{\mathcal{H}}(q,\kappa,p_q,p_{\kappa}) := \tilde{\mathcal{H}}^{\mathcal{E}}(q,\kappa,p_q,p_{\kappa},\bar{u})$ is called a (locally) optimal new control Hamiltonian (also control-independent new control Hamiltonian).
\end{definition}

    \begin{example}[Cont.~of Example \ref{exp:vec_example}]\hphantom{.}
    \label{exp:vec_example_cont_bis}
        \begin{enumerate}
            \item \label{itm:vec_example_cont_bis_1} Focusing on the case $\kappa \neq 0$, we have the fiber derivatives
            \begin{align*}
                \mathbb{F}\tilde{\mathcal{L}}^{\mathcal{E}}(q,\kappa,v_q,v_{\kappa},u) &= (q,\kappa, v_{\kappa} + \kappa (D_2 f_0 + D_2 f_3 u^3) - D_2 k, v_{q})\\
                \mathbb{F}\tilde{\mathcal{L}}(q,\kappa,v_q,v_{\kappa}) &= (q,\kappa, v_{\kappa} + \kappa D_2 f_0 - D_2 k, v_{q})
            \end{align*}
            as well as the new control Hamiltonians
            \begin{align*}
                \tilde{\mathcal{H}}^{\mathcal{E}}(q,\kappa,p_q,p_{\kappa},u) &= p_q p_{\kappa} - \kappa [f_0(q,p_{\kappa}) + f_3(q,p_{\kappa})u^3] + k(q,p_{\kappa})\\
                \tilde{\mathcal{H}}(q,\kappa,p_q,p_{\kappa}) &= p_q p_{\kappa} - \kappa f_0(q,p_{\kappa}) + k(q,p_{\kappa})
            \end{align*}
            \item \label{itm:vec_example_cont_bis_2} Using matrix notation, we get
            \begin{align*}
                \mathbb{F}\tilde{\mathcal{L}}^{\mathcal{E}}(q,\kappa,v_q,v_{\kappa},u) &= \left(q,\kappa, v_{\kappa} + (D_2 f_0 + D_2 f_1(u))^{\top}\kappa - \frac{1}{2} D_2 \mathrm{g}_{\mathcal{E}}(u,u), v_q\right)\\
                \mathbb{F}\tilde{\mathcal{L}}(q,\kappa,v_q,v_{\kappa}) &= \left(q,\kappa, v_{\kappa} + (D_2 f_0)^{\top}\kappa + \frac{1}{2} D_2 \mathrm{g}_{\mathcal{E}}^{-1}(f_1^{\top} \kappa,f_1^{\top} \kappa), v_q\right),
            \end{align*}
            while the Hamiltonians are
            \begin{align*}
                \tilde{\mathcal{H}}^{\mathcal{E}}(q,\kappa,p_q,p_{\kappa},u) &= p_{q}^{\top} p_{\kappa} - \kappa^{\top} [f_0(q,p_{\kappa}) + f_1(q,p_{\kappa})(u)] + \frac{1}{2}\mathrm{g}_{\mathcal{E}}(q,p_{\kappa})(u,u)\\
                \tilde{\mathcal{H}}(q,\kappa,p_q,p_{\kappa}) &= p_{q}^{\top} p_{\kappa} - \kappa^{\top} f_0(q,p_{\kappa}) - \frac{1}{2}\mathrm{g}_{\mathcal{E}}^{-1}(q,p_{\kappa})\left((f_1(q,p_{\kappa}))^{\top} \kappa,f_1(q,p_{\kappa}))^{\top} \kappa\right)\,.
            \end{align*}
        \end{enumerate}
    \end{example}

\subsection{Boundary costs}
In the process of integration by parts that led us from \eqref{eq:functional_J2} to \eqref{eq:functional_J3}, additional boundary costs, i.e. the initial and final costs
\begin{equation*}
    \kappa(T)^{\top} \dot{q}(T) - \kappa(0)^{\top} \dot{q}(0),
\end{equation*}
appear in the formulation. At first glance, the purpose and / or interpretation of these terms is difficult to parse.\\

Obviously they are necessary in order to have $\tilde{J}_3 = \tilde{J}_2$, and consequently, $\tilde{J}_1 = \tilde{J}_2 = \tilde{J}_3$ over optima. In fact, these equalities hold for curves satisfying the state and adjoint dynamics for a common $u$ compatible with the boundary conditions under the costate identifications provided earlier even if $u$ does not fulfill the maximization condition. A clearer picture begins to form precisely when we evaluate these functionals over curves these conditions.\\


Let us briefly consider the case of Lagrangian mechanics for a hyperregular Lagrangian. Consider the Lagrangian action \eqref{eq:Lagrangian_action}, and let $q_{a,b}: [t_a, t_b] \to \mathcal{Q}$ be a solution of the corresponding Euler-Lagrange equations with boundary values $q_a = q_{a,b}(t_a)$ and $q_b = q_{a,b}(t_b)$. When we evaluate the action over $q_{a,b}$, we define a new function, frequently denoted with the same letter as the action itself:
\begin{equation}
    \label{eq:Jacobi_solution}
    S(q_a,q_b) = \int_{t_a}^{t_b} L(q_{a,b}(t),\dot{q}_{a,b}(t)) \, dt.
\end{equation}
$S(q_a,q_b)$ is the Jacobi solution to the (in this case time-independent) Hamilton-Jacobi equation
\begin{equation*}
    E = H\left(q_a,-D_1 S(q_a,q_b)\right) = H\left(q_b, D_2 S(q_a,q_b)\right),
\end{equation*}
where $H$ is the Hamiltonian corresponding to $L$.  \eqref{eq:Jacobi_solution} behaves as a generating function of canonical transformations, i.e. symplectomorphisms, of the first kind \cite{Goldstein02}, meaning that
\begin{align*}
    D_1 S(q_a,q_b) &= - D_{2} L(q_{a,b}(t_a),\dot{q}_{a,b}(t_a)) = - p_a^{\top}\,,\\
    D_2 S(q_a,q_b) &= D_2 L(q_{a,b}(t_b),\dot{q}_{a,b}(t_b)) = p_b^{\top}\,.
\end{align*}
Their name stems from the fact that the induced map $\varphi: T^*\mathcal{Q} \to T^*\mathcal{Q}$, $(q_a,p_a) \mapsto (q_b,p_b)$ is indeed a symplectomorphism. These are particularly important for the formulation of the equivalence between continuous and discrete mechanics and the generation of variational integrators since \eqref{eq:Jacobi_solution} defines the \emph{so-called} exact discrete Lagrangian \cite{MarsdenWest01}. In \cite{Goldstein02} we also find generating functions of second, $S_2(p_a,q_b)$, third, $S_3(q_b,p_a)$, and fourth kind
\begin{equation*}
    S_4(p_a,p_b) + p_b^{\top} q_b - p_a^{\top} q_a = S(q_a,q_b)\,.
\end{equation*}
For these latter ones, we have that
\begin{align*}
    D_1 S_4(p_a,p_b) &= q_a\,,\\
    D_2 S_4(p_a,p_b) &= -q_b\,.
\end{align*}

Going back to the optimal control setting, one can proceed in a similar manner \cite{ParkScheeres04}. Let us consider a curve $(x_{a,b}, u_{a,b}) = (q_{a,b,},v_{a,b}, u_{a,b}) : [t_a,t_b] \to \mathcal{E}$, with $t_a < t_b \in (0,T)$, satisfying the state dynamics resulting from $\tilde{J}_1$ together with $(q_a,v_a) = x_{a,b}(t_a)$ and $(q_b,v_b) = x_{a,b}(t_b)$. Thus, one can generate a function
\begin{equation*}
    S_C^{\mathcal{E}}(q_a,v_a,q_b,v_b,u_{a,b}) = \int_{t_a}^{t_b} C(x_{a,b}(t),u_{a,b}(t)) \, dt\,,
\end{equation*}
which may be interpreted as a control-dependent generating function of first kind for the optimal control problem.
\emph{A priori}, the adjoint variables play no role in this definition. Nevertheless, we have that
\begin{equation*}
    \int_{t_a}^{t_b} C(x_{a,b}(t),u_{a,b}(t)) \, dt = \int_{t_a}^{t_b} \left[ C(x_{a,b}(t),u_{a,b}(t)) + \left\langle \lambda_{a,b}(t), \dot{x}_{a,b}(t) - X(x_{a,b}(t),u_{a,b}(t)) \right\rangle\right] \, dt
\end{equation*}
for any curve $\lambda_{a,b} = (\lambda_q^{a,b},\lambda_v^{a,b})$ over $x_{a,b}$. With this, one can show through variation 
that
\begin{align*}
    D_1 S_C^{\mathcal{E}}(q_a,v_a,q_b,v_b,u_{a,b}) &= -(\lambda_q^a)^{\top}\,,\\
    D_2 S_C^{\mathcal{E}}(q_a,v_a,q_b,v_b,u_{a,b}) &= -(\lambda_v^a)^{\top}\,,\\
    D_3 S_C^{\mathcal{E}}(q_a,v_a,q_b,v_b,u_{a,b}) &= (\lambda_q^b)^{\top}\,,\\
    D_4 S_C^{\mathcal{E}}(q_a,v_a,q_b,v_b,u_{a,b}) &= (\lambda_v^b)^{\top}\,,
\end{align*}
if and only if $\lambda_{a,b}$ is assumed to satisfy the adjoint dynamics resulting from $\tilde{J}_1$ over $[t_a,t_b]$ together with $(\lambda_q^a,\lambda_v^a) = (\lambda_{q}^{a,b}(t_a), \lambda_{v}^{a,b}(t_a))$ and $(\lambda_q^b,\lambda_v^b) = (\lambda_{q}^{a,b}(t_b), \lambda_{v}^{a,b}(t_b))$.\\

Let us now consider curves $(y_{a,b}, u_{a,b}) = (q_{a,b}, \kappa_{a,b}, u_{a,b}) = (q_{a,b}, \lambda_v^{a,b}, u_{a,b})$  satisfying the state and adjoint dynamics resulting from $\tilde{J}_3$ over $[t_a,t_b]$ together with $(q_a,\kappa_a) = y_{a,b}(t_a)$ and $(q_b,\kappa_b) = y_{a,b}(t_b)$. If we define the function
\begin{equation*}
    S_{\tilde{\mathcal{L}}}^{\mathcal{E}}(q_a,\kappa_a,q_b,\kappa_b,u_{a,b}) = \int_{t_a}^{t_b} \tilde{\mathcal{L}}^{\mathcal{E}}(y_{a,b}(t),u_{a,b}(t)) \, dt\,.
\end{equation*}
we find that
\begin{align*}
    D_1 S_{\tilde{\mathcal{L}}}^{\mathcal{E}}(q_a,\kappa_a,q_b,\kappa_b,u_{a,b}) &= - (p_q^{a})^{\top} = (\lambda_q^a)^{\top}\,,\\
    D_2 S_{\tilde{\mathcal{L}}}^{\mathcal{E}}(q_a,\kappa_a,q_b,\kappa_b,u_{a,b}) &= - (p_{\kappa}^{a})^{\top} = - (v_a)^{\top} = - (v_{a,b}(t_a))^{\top}\,,\\
    D_3 S_{\tilde{\mathcal{L}}}^{\mathcal{E}}(q_a,\kappa_a,q_b,\kappa_b,u_{a,b}) &= (p_q^{b})^{\top} = -(\lambda_q^b)^{\top}\,,\\
    D_4 S_{\tilde{\mathcal{L}}}^{\mathcal{E}}(q_a,\kappa_a,q_b,\kappa_b,u_{a,b}) &= (p_{\kappa}^{b})^{\top} = (v_b)^{\top} = (v_{a,b}(t_b))^{\top}\,.
\end{align*}
This has already been explored in \cite{Konopik25a} for a restricted class of $\tilde{\mathcal{L}}^{\mathcal{E}}$. From the equivalence of $\tilde{J}_1$, $\tilde{J}_2$ and $\tilde{J}_3$ we get that
\begin{equation*}
    \kappa_b^{\top} v_b - \kappa_a^{\top} v_a - S_{\tilde{\mathcal{L}}}^{\mathcal{E}}(q_a,\kappa_a,q_b,\kappa_b,u_{a,b}) = S_{C}^{\mathcal{E}}(q_a,v_a,q_b,v_b,u_{a,b}),
\end{equation*}
and so
\begin{align*}
    D_{q_a}\left[ \kappa_b^{\top} v_b - \kappa_a^{\top} v_a - S_{\tilde{\mathcal{L}}}^{\mathcal{E}}(q_a,\kappa_a,q_b,\kappa_b,u_{a,b})\right] &= -(\lambda_q^a)^{\top}\,,\\
    D_{v_a}\left[ \kappa_b^{\top} v_b - \kappa_a^{\top} v_a - S_{\tilde{\mathcal{L}}}^{\mathcal{E}}(q_a,\kappa_a,q_b,\kappa_b,u_{a,b})\right] &= - \kappa_a^{\top} = - (\lambda_v^a)^{\top}\,,\\
    D_{\kappa_a}\left[ \kappa_b^{\top} v_b - \kappa_a^{\top} v_a - S_{\tilde{\mathcal{L}}}^{\mathcal{E}}(q_a,\kappa_a,q_b,\kappa_b,u_{a,b})\right] &= 0\,,\\
    D_{q_b}\left[ \kappa_b^{\top} v_b - \kappa_a^{\top} v_a - S_{\tilde{\mathcal{L}}}^{\mathcal{E}}(q_a,\kappa_a,q_b,\kappa_b,u_{a,b})\right] &= (\lambda_q^b)^{\top}\,,\\
    D_{v_b}\left[ \kappa_b^{\top} v_b - \kappa_a^{\top} v_a - S_{\tilde{\mathcal{L}}}^{\mathcal{E}}(q_a,\kappa_a,q_b,\kappa_b,u_{a,b})\right] &= \kappa_{b}^{\top} = (\lambda_{v}^b)^{\top}\,,\\
    D_{{\kappa}_b}\left[ \kappa_b^{\top} v_b - \kappa_a^{\top} v_a - S_{\tilde{\mathcal{L}}}^{\mathcal{E}}(q_a,\kappa_a,q_b,\kappa_b,u_{a,b})\right] &= 0\,.
\end{align*}
Therefore, the additional boundary costs can be interpreted as an interfacing or translation layer that allows us to pass from one formulation to the other, i.e. one set of variables to the other, so that the resulting canonical momenta at the boundaries are consistent. Moreover, comparing these boundary terms with the additional terms that appear when defining generating functions of the fourth kind, we see that when using $S_{\tilde{\mathcal{L}}}^{\mathcal{E}}$ we are working with a mixed generating function of first and fourth kind: first on $q$, fourth on $v$ since $\kappa = \lambda_v$ are their associated canonical momenta.

\begin{remark}
    When evaluated over $u_{a,b}$ satisfying the corresponding maximization condition, $S_{C}^{\mathcal{E}}$ and $S_{\tilde{\mathcal{L}}}^{\mathcal{E}}$ become standard generating functions. Therefore, the induced transformations $(q_a,v_a,\lambda_q^a,\lambda_v^a) \mapsto (q_b,v_b,\lambda_q^b,\lambda_v^b)$ and $(q_a,\kappa_a,p_q^a,p_{\kappa}^a) \mapsto (q_b,\kappa_b,p_q^b,p_{\kappa}^b)$ are symplectic on $T^*T\mathcal{Q}$ and $T^*T^*\mathcal{Q}$ respectively, which can be important in numerical applications. They are, however, equivalent, as will be clarified in Section \ref{sec:tulcyjew_and_PMP}.
\end{remark}

\subsection{Reformulation for Lagrangian systems}
\label{ssec:OptimalControlLagrangian}

    As we saw in Section~\ref{sssec:force_controlled EL_eqs}, given a regular force-controlled Lagrangian system, its equations of motion \eqref{eq:force_controlled_EL} are expressible as controlled SODEs. Therefore, the former constructions are readily applicable and, in particular, given a running cost we may immediately construct an associated new control Lagrangian as per Definition~\ref{def:new_control_Lagrangian}.

    However, as we previously mentioned in Remark~\ref{rmk:Lagrangian_SODE_num}, it is not generally advisable to do so, particularly when the aim is to perform numerical computations. Whenever the underlying controlled SODE is known to originate from forced Euler-Lagrange equations, it is best to work with the latter.

    Let us assume that is the case. Also, in order to simplify the geometric picture, let us assume the underlying Lagrangian system is not only regular but hyperregular. Then \eqref{eq:functional_J2} transforms into
    \begin{align}
            &\hat{J}_2(z,u) = \phi(q(T),\dot{q}(T)) \label{eq:functional_J2_EL}\\
            &\quad + \int_{0}^T \left\lbrace C(q(t),\dot{q}(t),u(t)) + \left[ \frac{d}{dt}\left(D_2 L(q(t),\dot{q}(t))\right) - D_1 L(q(t),\dot{q}(t)) -f_L^{\mathcal{E}}(q(t),\dot{q}(t),u(t)) \right] \, \xi(t) \right\rbrace dt\,.\nonumber
    \end{align}
    where $z = (q,\xi)$ is a curve on $T\mathcal{Q}$. If we expand the force-controlled Euler-Lagrange equations and express them as explicit controlled SODEs, then, this reduces to \eqref{eq:functional_J2} under the substitution $\kappa^{\top} = D_{2 2} L(q,v) \, \xi$.

    \begin{remark}
        Notice that the force-controlled Euler-Lagrange equations behave as the components of a semi-basic form. This already implies that $\xi$ must have vectorial character. Moreover, the regularity of the force-controlled Lagrangian system implies that $D_{2 2} L(q,v)$ plays the role of a (possibly pseudo-Riemannian) metric on $V T\mathcal{Q}$, the vertical bundle of $T \mathcal{Q}$, conjugate to the space of semi-basic forms. This underlines the same fact in the relation between the covector $\kappa$ and the vector $\xi$.
    \end{remark}

    Let us finally apply integration by parts as in \eqref{eq:functional_J3}. This leads to the new functional
    \begin{align}
		&\tilde{J}_4(y,u) = \phi(q(T),\dot{q}(T)) + D_2 L(q(T),\dot{q}(T)) \, \xi(T) - D_2 L(q(0),\dot{q}(0)) \, \xi(0) \label{eq:functional_J4}\\
		&\quad+ \int_{0}^T \left[ C(q(t),\dot{q}(t),u(t)) - D_2 
 L(q(t),\dot{q}(t)) \, \dot{\xi}(t) - \left( D_1 
 L(q(t),\dot{q}(t)) + 
 f_L^{\mathcal{E}}(q(t),\dot{q}(t),u(t))\right) \, \xi(t) \right] dt\,.\nonumber
    \end{align}

This leads us once more to provide the following

    \begin{definition}
    \label{def:new_control_Lagrangian_EL}
        Let $(\mathcal{Q},\mathcal{E},L,f_L^{\mathcal{E}})$ be a regular force-controlled Lagrangian system and $C: \mathcal{E} \to \mathbb{R}$ a running cost function. Then, $\tilde{\mathcal{L}}^\mathcal{E}_L: TT\mathcal{Q} \oplus_{T\mathcal{Q}}^{\kappa} \mathcal{E} \to \mathbb{R}$,
        \begin{equation*}
            \tilde{\mathcal{L}}^\mathcal{E}_L(q,\xi,v_q,v_{\xi},u) = D_2 L(q,v_q) \, v_{\xi} + \left[ D_1 L(q,v_q) + f_L^{\mathcal{E}}(q,v_q,u) \right] \, \xi - C(q,v_q,u)\,,
        \end{equation*}
        is called a new control Lagrangian for the force-controlled Lagrangian system.
    \end{definition}
    We leave the formal definition of $TT\mathcal{Q} \oplus_{T\mathcal{Q}}^{\kappa} \mathcal{E}$ for Section \ref{sec:tulcyjew_and_PMP}.\\

     Once more, one can check that the necessary optimality conditions for $\tilde{J}_4$ are given by
            \begin{itemize}
                \item (adjoint dynamics) $0 = D_1 \tilde{\mathcal{L}}_L^\mathcal{E}(q(t),\xi(t),\dot{q}(t),\dot{\xi}(t),u(t)) - \frac{d}{dt}\left( D_3 \tilde{\mathcal{L}}_L^\mathcal{E}(q(t),\xi(t),\dot{q}(t),\dot{\xi}(t),u(t)) \right)$\\
                \hphantom{(adjoint dynamics) 0} $= -D_1 C + D_{2 1} L \,\dot{\xi} + (D_{1 1}L + D_1 f_L^{\mathcal{E}}) \, \xi + \frac{d}{dt} \left[ D_2 C - D_{2 2} L \, \dot{\xi} - (D_{1 2}L + D_2 f_L^{\mathcal{E}}) \, \xi \right]$
                \item (state dynamics) $0 = D_2 \tilde{\mathcal{L}}_L^\mathcal{E}(q(t),\xi(t),\dot{q}(t),\dot{\xi}(t),u(t)) - \frac{d}{dt}\left( D_4 \tilde{\mathcal{L}}_L^\mathcal{E}(q(t),\xi(t),\dot{q}(t),\dot{\xi}(t),u(t)) \right)$\\
                \hphantom{(state dynamics) 0}
                $= D_1 L + f_L^{\mathcal{E}} - \frac{d}{dt}\left( D_2 L \right)$,
                \item (maximization) $0 = D_5 \tilde{\mathcal{L}}_L^\mathcal{E}(q(t),\xi(t),\dot{q}(t),\dot{\xi}(t),u(t))$\,\\
                \hphantom{(maximization) 0}
                $= D_3 C - D_3 f_L^{\mathcal{E}}\, \xi$,
                \item (transversality) $D_{2 2} L(q(T),\dot{q}(T)) \xi(T) = -D_2 \phi(q(T),\dot{q}(T))$,\\
                \hphantom{(transversality)} $D_{2 2} L(q(T),\dot{q}(T)) \dot{\xi}(T) = D_1 \phi(q(T),\dot{q}(T))$\\
                \hphantom{(transversality) $D_{2 2} L(q(T),\dot{q}(T)) \dot{\xi}(T)$} $+\, D_2 C(q(T),\dot{q}(T),u(T)) - \xi(T) D_2 f_L^{\mathcal{E}}(q(T),\dot{q}(T),u(t))$.
        \end{itemize}

\begin{remark}
    Notice that the adjoint dynamics take the form of a modified Jacobi equation for the force-controlled Euler-Lagrange equations. In it, the Euler-Lagrange operator applied to the running cost $C$ plays the role of a forcing of sorts.
\end{remark}

\begin{example}
    \label{exp:low_thrust}
    Let us take the low thrust orbital transfer example in \cite{Leyendecker2024new}. There $\mathcal{Q} = \mathbb{R}^2 \setminus\{0\}$, $\mathcal{N} = \mathbb{R}$, $L: T(\mathbb{R}^2 \setminus\{0\}) \to \mathbb{R}$, using polar coordinates
    \begin{equation*}
        L(q,v_q) = L(r,\varphi,v_r,v_\varphi) = \frac{1}{2} m (v_r^2 + r^2 v_\varphi^2)+ \gamma \frac{M m}{r}
    \end{equation*}
    and $f_L^{\mathcal{E}} = m r u d \varphi$. The resulting equations of motion are
    \begin{align*}
             \frac{d}{dt}( m\dot{r}) &=  mr\dot{\varphi}^2-\frac{\gamma M m}{r^2}\\
             \frac{d}{dt}( mr^2\dot{\varphi}) &=  mru.
    \end{align*}
    The cost function $\mathcal{C} : \mathcal{E} \to \mathbb{R}$ takes the particularly simple form
    \begin{equation*}
        C(q,v_q,u) = C(r,\varphi,v_r,v_\varphi,u) = \frac{1}{2} u^2,
    \end{equation*}
    so with the conventions established here
    \begin{align*}
        \tilde{\mathcal{L}}^{\mathcal{E}}(q,\kappa,v_q,v_{\kappa},u) &= \tilde{\mathcal{L}}^{\mathcal{E}}(r,\varphi,\kappa_r,\kappa_{\varphi},v_r,v_\varphi,v_{\kappa_r},v_{\kappa_{\varphi}},u)\\
        &= v_{\kappa_r} v_r + v_{\kappa_{\varphi}} v_\varphi + \kappa_r \left( r v_{\varphi}^2 - \frac{\gamma M }{r^2}\right) + \kappa_{\varphi}\left(-\frac{2 v_{r} v_{\varphi}}{r} + \frac{u}{r}\right) - \frac{1}{2} u^2\,.\\
        \tilde{\mathcal{L}}_L^{\mathcal{E}}(q,\xi,v_q,v_{\xi},u) &= \tilde{\mathcal{L}}_L^{\mathcal{E}}(r,\varphi,\xi_r,\xi_{\varphi},v_r,v_\varphi,v_{\xi_r},v_{\xi_{\varphi}},u)\\
        &= m v_{\xi_r} v_r + m r^2 v_{\xi_{\varphi}} v_\varphi + m \left( r v_{\varphi}^2 - \frac{\gamma M }{r^2}\right) \xi_r + m r u \xi_{\varphi} - \frac{1}{2} u^2\,.
    \end{align*}
    One can check that the resulting state and adjoint dynamics are the same under the identification $\kappa_r = m \xi_r$ and $\kappa_{\varphi} = m r^2 \xi_{\varphi}$.
\end{example}

Consider now the fiber derivative induced by $\tilde{\mathcal{L}}^{\mathcal{E}}_L$, namely, $\mathbb{F}\tilde{\mathcal{L}}^{\mathcal{E}}_L: TT\mathcal{Q} \oplus_{T\mathcal{Q}}^{\kappa} \mathcal{E} \to T^*T\mathcal{Q} \oplus_{T\mathcal{Q}}^{\alpha,L} \mathcal{E}$, where $T^*T\mathcal{Q} \oplus_{T\mathcal{Q}}^{\alpha,L} \mathcal{E}$ will be defined in Section \ref{sec:tulcyjew_and_PMP},
    \begin{equation*}
        \mathbb{F}\tilde{\mathcal{L}}^{\mathcal{E}}_L( q,  \xi, v_q, v_{\xi}, u) = \left(q, \xi, \varpi_{q} = D_{3}\tilde{\mathcal{L}}^{\mathcal{E}}_L( q,  \xi, v_q, v_{\xi}, u), \varpi_{\xi} = D_{4}\tilde{\mathcal{L}}^{\mathcal{E}}_L( q,  \xi, v_q, v_{\xi}, u), u\right),
    \end{equation*}

More explicitly,
\begin{align*}
    \varpi_q^{\top} &= D_{2 2} L(q,v_{q}) \, v_{\xi} + \left[ D_{1 2} L(q,v_{q}) + D_2 f_L^{\mathcal{E}}(q,v_{q},u)\right] \, \xi - D_2 C(q,v_{q},u)\,,\\
    \varpi_{\xi}^{\top} &= D_2 L(q,v_{q})\,.
\end{align*}

With this, and under our hyperregularity and smoothness assumptions it is immediate to check that this relationship is smoothly invertible, which provides us with the analogue of Proposition~\ref{prp:hyperregularity_new_Lagrangian}:

\begin{proposition}
    When $u$ is regarded as a parameter, if $L$ is hyperregular, then $\tilde{\mathcal{L}}^{\mathcal{E}}_L$ is also hyperregular.
\end{proposition}

Just like in the controlled SODE case, we can also define an associated energy function,
    \begin{equation*}
        E_{\tilde{\mathcal{L}}^{\mathcal{E}}_L} = \triangle \tilde{\mathcal{L}}^{\mathcal{E}}_L - \tilde{\mathcal{L}}^{\mathcal{E}}_L\,,
    \end{equation*}
    with Liouville field $\triangle = v_q^i \partial_{v_q^i} + v_{\xi}^i \partial_{v_{\xi}^i}$ in this case.\\
    
Using the hyperregularity of $\tilde{\mathcal{L}}^{\mathcal{E}}_L$ we can also define a new control Hamiltonian in this setting.

    \begin{definition}
        Let $\tilde{\mathcal{L}}^{\mathcal{E}}_L$ be a new control Lagrangian for a hyperregular force-controlled Lagrangian system $(\mathcal{Q},\mathcal{E},L,f_L^{\mathcal{E}})$ with running cost $C$. Then, the Legendre transform of this new control Lagrangian, i.e.
        $\tilde{\mathcal{H}}^{\mathcal{E}}_L = (E_{\tilde{\mathcal{L}}^{\mathcal{E}}_L} \circ (\mathbb{F}\tilde{\mathcal{L}}^{\mathcal{E}}_L)^{-1}) : T^*T\mathcal{Q} \oplus_{T\mathcal{Q}}^{\alpha,L} \mathcal{E} \to \mathbb{R}$,
        \begin{equation*}
             \tilde{\mathcal{H}}^{\mathcal{E}}_L (q,\xi,\varpi_q,\varpi_{\xi},u) = \varpi_{q}^{\top} v_q(q,\varpi_{\xi}) - \left[ D_1 L(q,v_q(q,\varpi_{\xi})) + f_L^{\mathcal{E}}(q,v_q(q,\varpi_{\xi}),u) \right] \xi + C(q,v_q(q,\varpi_{\xi}),u),
        \end{equation*}
        where $v_q(q,\varpi_{\xi})$ is the solution of $\varpi_{\xi}^{\top} = D_2 L(q,v_q)$,
        is called a new control Hamiltonian for the force-controlled Lagrangian system.
    \end{definition}
    As in the controlled SODE case, we may also provide the following
    \begin{definition}
        Assume $\bar{u}(t)$ is such that the maximization condition is satisfied. Then,
        $\tilde{\mathcal{L}}_L(q,\xi,v_q,v_{\xi}) := \tilde{\mathcal{L}}_L^{\mathcal{E}}(q,\xi,v_q,v_{\xi},\bar{u})$ and $\tilde{\mathcal{H}}_L(q,\xi,\varpi_q,\varpi_{\xi}) := \tilde{\mathcal{H}}_L^{\mathcal{E}}(q,\xi,\varpi_q,\varpi_{\xi},\bar{u})$ are called a (locally) optimal new control Lagrangian and Hamiltonian for the force-controlled Lagrangian system respectively.
    \end{definition}

    Now, since $L$ is hyperregular, $\mathbb{F}L: T\mathcal{Q} \to T^*\mathcal{Q}$ is a diffeomorphism. This also implies that $L$ has a corresponding Hamiltonian, $H: T^*Q \to \mathbb{R}$, $H = (E_L \circ (\mathbb{F}L)^{-1})$,
    \begin{equation*}
        H(q,p_q) = p_q^{\top} v(q,p_q) - L(q, v(q,p_q)).
    \end{equation*}
    where $v_q(q,p_q)$ is the solution of $p_q^{\top} = D_2 L(q,v_q)$. This velocity, however, can be rewritten in terms of the Hamiltonian itself, as $v_q(q,p_q) = D_2 H(q,p_q)$. One also finds that $D_1 H(q,p_q) = - D_1 L(q,v_q(q,p_q))$.

    One can also consider a new control space $(\mathcal{F},\pi^{\mathcal{F}},T^*\mathcal{Q})$, provided by any vector bundle isomorphism over $\mathbb{F}L$,
    \begin{center}
        \begin{tikzcd}
        \mathcal{E} \arrow[rr, "\chi_{\mathbb{F}L}"] \arrow[dd, "\pi^{\mathcal{E}}"'] &  & \mathcal{F} \arrow[dd, "\pi^{\mathcal{F}}"] \\
                                                                                      &  &                                             \\
        T\mathcal{Q} \arrow[rr, "\mathbb{F}L"]                                        &  & T^*\mathcal{Q}                             
        \end{tikzcd}
    \end{center}
    
    and define $C^H = C \circ \chi_{\mathbb{F}L}^{-1}$, and $f_H^{\mathcal{F}} = f_L^{\mathcal{E}} \circ \chi_{\mathbb{F}L}^{-1}$. In general $\chi_{\mathbb{F}L}(q,v_q,u) = (q,p_q = D_2 L(q,v_q), w = \chi(q,v_q) u)$. Let us choose, for simplicity $\chi_{\mathbb{F}L}(q,v_q,u) = (q,p_q = D_2 L(q,v_q),u)$, $\chi_{\mathbb{F}L}^{-1}(q,p_q,u) = (q,v_q = D_2 H(q,p_q),u)$ to perform the following computations. Thus,
    \begin{align*}
        C_{H}(q,p_q,u) &= C(q,D_2 H(q,p_q),u)\\
        f_H^{\mathcal{F}}(q,p_q,u) &= f_L^{\mathcal{E}}(q,D_2 H(q,p_q),u).
    \end{align*}
    The equations of motion of the system are now forced Hamilton's equations, which take the local form
    \begin{subequations}
    \begin{align}
        \dot{q}(t) &= D_2 H(q,p_q),\label{eq:forced_Hamiltonian_eqs_1}\\
        \dot{p}_q(t) &= -D_1 H(q,p_q) + f_H^{\mathcal{F}}(q,p_q,u),\label{eq:forced_Hamiltonian_eqs_2}
    \end{align}
    \label{eq:forced_Hamiltonian_eqs}
    \end{subequations}
    or, in invariant form
    \begin{equation*}
        \imath_{X_H} \omega = dH + \eta^{\mathcal{F}},
    \end{equation*}
    where $X_H$ is the Hamiltonian vector field corresponding to \eqref{eq:forced_Hamiltonian_eqs} and $\eta^{\mathcal{F}} \in \Gamma_{\mathrm{loc}}(\mathcal{F},T^*T^*\mathcal{Q})$, interpreted as a semibasic forcing form $\eta^{\mathcal{F}} = (f_H^{\mathcal{F}})_i(q,p_q,u) \, dq^i$. We say that $(T^*\mathcal{Q},\mathcal{F},H,f_H^{\mathcal{F}})$ is a force-controlled Hamiltonian system on $T^*\mathcal{Q}$. Realizing that $\varpi_{\xi} = p_q$ and plugging these definitions into $\tilde{\mathcal{H}}^{\mathcal{E}}_L$, we obtain
    \begin{equation}
        \label{eq:new_old_Hamiltonian}
        \tilde{\mathcal{H}}^{\mathcal{E}}_L(q,\xi,\varpi_q,p_q,u) = \varpi_{q}^{\top} D_2 H(q,p_q) + \left[ D_1 H(q,p_q) - f_H^{\mathcal{F}}(q,p_q,u) \right] \xi + C_H(q,p_q,u).
    \end{equation}
    But, had we started from the force-controlled Hamiltonian system from the very beginning, Pontryagin's control Hamiltonian \eqref{eq:Pontryagins_control_Hamiltonian} would take the form
    \begin{align*}
        \mathcal{H}_{-1}(q,p_q,\lambda_q,\lambda_p,u) &= \left\langle (\lambda_q, \lambda_p),(D_2 H(q,p_q), -D_1 H(q,p_q) + f_H^{\mathcal{F}}(q,p_q,u)) \right\rangle_{T^*\mathcal{Q}} - C_H(q,p_q,u)\\
       &= \lambda_q^\top D_2 H(q,p_q) - \left[ D_1 H(q,p_q) - f_H^{\mathcal{F}}(q,p_q,u) \right] \lambda_p - C_H(q,p_q,u)
    \end{align*}
    However, under the identification $\varpi_q = -\lambda_q$, $\xi = \lambda_p$, we see that this coincides with $-\tilde{\mathcal{H}}^{\mathcal{E}}_L$. We will shed some light on this result in the following section.
    

    \begin{example}[Cont.~of Example \ref{exp:low_thrust}]
    Let $\chi_{\mathbb{F}L}(q,v_q,u) = (q,p_q=D_2 L(q,v_q),u)$. Then,
    \begin{align*}
        H(q,p_q) = H(r,\varphi,p_r,p_\varphi) &= \frac{1}{2 m} \left( p_r^2 + \frac{p_\varphi^2}{r^2}\right) - \gamma \frac{M m}{r}\\
        f_H^{\mathcal{F}}(q,p_q,u) = f_H^{\mathcal{F}}(r,\varphi,p_r,p_\varphi,u) &= mr u\\
        C_H(q,p_q,u) = C_H(r,\varphi,p_r,p_\varphi,u) &= \frac{1}{2} u^2\,.
    \end{align*}
    Inverting $\mathbb{F}L$ we get that $v_q(q,p_q) = (p_r/m, p_{\varphi}/(m r^2))$, so
    \begin{align*}
        \tilde{\mathcal{H}}_L^{\mathcal{E}}(q,\xi,\varpi_q,\varpi_{\xi},u) &= \tilde{\mathcal{H}}_L^{\mathcal{E}}(r,\varphi,\xi_r,\xi_{\varphi},\varpi_r, \varpi_{\varphi},\varpi_{\xi_r}, \varpi_{\xi_{\varphi}},u)\\
        &= \varpi_r \frac{\varpi_{\xi_r}}{m} + \varpi_{\varphi} \frac{\varpi_{\xi_{\varphi}}}{m r^2} - \left( \frac{{\varpi_{\xi_{\varphi}}}^2}{m r^3} - \frac{\gamma M }{r^2}\right) \xi_r + m r u \xi_{\varphi} + \frac{1}{2} u^2\\
        \mathcal{H}_{-1}(q,p_q,\lambda_q,\lambda_{p},u) &= \mathcal{H}_{-1}(r,\varphi,p_r,p_{\varphi},\lambda_r, \lambda_{\varphi},\lambda_{p_r}, \lambda_{p_{\varphi}},u)\\
        &= \lambda_r \frac{p_r}{m} + \lambda_{\varphi} \frac{p_{\varphi}}{m r^2} + \lambda_{p_r} \left( \frac{{\varpi_{\xi_{\varphi}}}^2}{m r^3} - \frac{\gamma M }{r^2}\right) - \lambda_{p_{\varphi}} m r u - \frac{1}{2} u^2\,.\\
    \end{align*}
    \end{example}
    %
    

\section{Tulczyjew's triple in optimal control of second-order systems}
\label{sec:tulcyjew_and_PMP}

The double bundles derived from $T\mathcal{Q}$ and $T^*\mathcal{Q}$, namely $TT\mathcal{Q}$, $T^*T\mathcal{Q}$, $TT^*\mathcal{Q}$ and $T^*T^*\mathcal{Q}$ have a rich geometric structure. In particular, the latter three are part of what is called Tulczyjew's triple. This is an isomorphic relation between these bundles, via two isomorphisms $\alpha_\mathcal{Q}: TT^*\mathcal{Q} \to T^*T\mathcal{Q}$ and $\beta_\mathcal{Q}: TT^*\mathcal{Q} \to T^*T^*\mathcal{Q}$ introduced in his papers \cite{TulczyjewHam,TulczyjewLag}.

\begin{center}
    \begin{tikzcd}
    T^*T\mathcal{Q} \arrow[rdd, "\pi_{T\mathcal{Q}}"'] &                                                 & TT^*\mathcal{Q} \arrow[ldd, "T\pi_{\mathcal{Q}}"'] \arrow[rdd, "\tau_{T^*\mathcal{Q}}"] \arrow[ll, "\alpha_{\mathcal{Q}}"'] \arrow[rr, "\beta_{\mathcal{Q}}"] &                                                 & T^*T^*\mathcal{Q} \arrow[ldd, "\pi_{T^*\mathcal{Q}}"] \\
                                                       &                                                 &                                                                                                                                                               &                                                 &                                                       \\
                                                       & T\mathcal{Q} \arrow[rdd, "\tau_{\mathcal{Q}}"'] &                                                                                                                                                               & T^*\mathcal{Q} \arrow[ldd, "\pi_{\mathcal{Q}}"] &                                                       \\
                                                       &                                                 &                                                                                                                                                               &                                                 &                                                       \\
                                                       &                                                 & \mathcal{Q}                                                                                                                                                   &                                                 &                                                      
    \end{tikzcd}
\end{center}

The former, $\alpha_{\mathcal{Q}}$, is in a sense dual to the canonical involution in $TT\mathcal{Q}$, $\kappa_{\mathcal{Q}}: TT\mathcal{Q} \to TT\mathcal{Q}$ \cite{Godbillon}, which codifies the fact that for all $f \in C^{\infty}(\mathbb{R}^2, \mathcal{Q})$, its derivatives commute, i.e.
\begin{equation*}
\kappa_\mathcal{Q} \left( \left.\frac{\mathrm{d}}{\mathrm{d} s}\right\vert_{s = 0} \left( \left.\frac{\mathrm{d}}{\mathrm{d} t}\right\vert_{t = 0} f(t,s)\right) \right) = \left.\frac{\mathrm{d}}{\mathrm{d} t}\right\vert_{t = 0} \left( \left.\frac{\mathrm{d}}{\mathrm{d} s}\right\vert_{s = 0} f(t,s)\right).
\end{equation*}
In adapted local coordinates, if $(q,v,X_q,X_v) \in TT\mathcal{Q}$, then $\kappa_{\mathcal{Q}}(q,v,X_q,X_v) = (q,X_q,v,X_v)$. The latter, $\beta_{\mathcal{Q}}$, is provided by the action of the canonical symplectic form on $T^*\mathcal{Q}$.\\

In adapted local coordinates, Tulcyjew's diffeomorphisms correspond to simple rearrangements of coordinates, namely if $(q,\kappa,v_q,v_{\kappa}) \in TT^*\mathcal{Q}$, then
\begin{align*}
    \alpha_{\mathcal{Q}}(q,\kappa,v_q,v_{\kappa}) &= (q,v_q,v_{\kappa},\kappa) \in T^*T\mathcal{Q}\\
    \beta_{\mathcal{Q}}(q,\kappa,v_q,v_{\kappa}) &= (q,\kappa,-v_{\kappa},v_q) \in T^*T^*\mathcal{Q}
\end{align*}

Similar constructions can be carried out in the spaces where we have been developing our theory. However, since these are sums of vector bundles, these pose the additional difficulty of needing to map correctly into each summand. Developing this framework is the purpose of the rest of this section.

\begin{definition}
    Let
    \begin{equation*}
        TT^*\mathcal{Q} \oplus_{T\mathcal{Q}}^\alpha \mathcal{E} = \left\lbrace (V,U) \in TT^*\mathcal{Q} \times \mathcal{E} \;\vert\; \pi_{T\mathcal{Q}} \circ \alpha_{\mathcal{Q}} (V) = \pi^{\mathcal{E}}(U) \right\rbrace
    \end{equation*}
    We say $TT^*\mathcal{Q} \oplus_{T\mathcal{Q}}^\alpha \mathcal{E}$ is the $\alpha_\mathcal{Q}$-twisted sum of $TT^*\mathcal{Q}$ and $\mathcal{E}$.
\end{definition}

Analogously to a Whitney sum, given $W \in  TT^*\mathcal{Q} \oplus_{T\mathcal{Q}}^\alpha \mathcal{E}$, we naturally get the structural projections $\mathrm{pr}_1^{\alpha}(W) = V$ and $\mathrm{pr}_2^{\alpha}(W) = U$ provided by the Cartesian product structure that defines the space.\\

In local adapted coordinates, if $V = (q,\kappa,v_q,v_{\kappa}) \in TT^*\mathcal{Q}$ and $U = (q,v_q,u)$, then, we label the corresponding point in $TT^*\mathcal{Q} \oplus_{T\mathcal{Q}}^\alpha \mathcal{E}$ by $W = (q,\kappa,v_q,v_{\kappa},u)$. Thus, $\alpha_\mathcal{Q}(q,\kappa,v_q,v_{\kappa}) = (q,v_q,v_{\kappa},\kappa)$ and $\pi_{TQ}(q,v_q,v_{\kappa},\kappa) = (q,v_q) = \pi^{\mathcal{E}}(q,v_q,u)$.\\

This structure allows us to define the diffeomorphism $\alpha_\mathcal{Q}^{\mathcal{E}}: TT^*\mathcal{Q} \oplus_{T\mathcal{Q}}^\alpha \mathcal{E} \to T^*T\mathcal{Q} \oplus_{T\mathcal{Q}} \mathcal{E}$, the analogue of $\alpha_\mathcal{Q}$ extended to the sum of vector bundles, that makes the following diagram commutative.

\begin{center}
    \begin{tikzcd}
    TT^*\mathcal{Q} \oplus_{T\mathcal{Q}}^{\alpha} \mathcal{E} \arrow[rrrrdd, "\mathrm{pr}^{\alpha}_2", bend left] \arrow[rrdddddd, "\mathrm{pr}^{\alpha}_1"', bend right] \arrow[rrdd, "\alpha_{\mathcal{Q}}^\mathcal{E}", shift left] &  &                                                                                                                                                                                               &  &                                              \\
                                                                                                                                                                                                                                        &  &                                                                                                                                                                                               &  &                                              \\
                                                                                                                                                                                                                                        &  & T^*T{\mathcal{Q}} \oplus_{T\mathcal{Q}} \mathcal{E} \arrow[rr, "\mathrm{pr}_2"] \arrow[dd, "\mathrm{pr}_1"'] \arrow[dd] \arrow[lluu, "(\alpha_{\mathcal{Q}}^{\mathcal{E}})^{-1}", shift left] &  & \mathcal{E} \arrow[ddd, "\pi^{\mathcal{E}}"] \\
                                                                                                                                                                                                                                        &  &                                                                                                                                                                                               &  &                                              \\
                                                                                                                                                                                                                                        &  & T^*T\mathcal{Q} \arrow[dd, "\alpha^{-1}_{\mathcal{Q}}"', shift right] \arrow[rrd, "\pi_{T\mathcal{Q}}"]                                                                                       &  &                                              \\
                                                                                                                                                                                                                                        &  &                                                                                                                                                                                               &  & T\mathcal{Q}                                 \\
                                                                                                                                                                                                                                        &  & TT^*\mathcal{Q} \arrow[rru, "T\pi_{\mathcal{Q}}"'] \arrow[uu, "\alpha_{\mathcal{Q}}"', shift right]                                                                                           &  &                                             
    \end{tikzcd}
\end{center}

\begin{definition}
    Let
    \begin{equation*}
        T^*T^*\mathcal{Q} \oplus_{T\mathcal{Q}}^\beta \mathcal{E} = \left\lbrace (P,U) \in T^*T^*\mathcal{Q} \times \mathcal{E} \;\vert\; T\pi_{\mathcal{Q}} \circ \beta^{-1}_{\mathcal{Q}} (P) = \pi^{\mathcal{E}}(U) \right\rbrace
    \end{equation*}
    We say $T^*T^*\mathcal{Q} \oplus_{T\mathcal{Q}}^\beta \mathcal{E}$ is the $\beta_\mathcal{Q}$-twisted sum of $T^*T^*\mathcal{Q}$ and $\mathcal{E}$.
\end{definition}

Given $\Pi \in T^*T^*\mathcal{Q} \oplus_{T\mathcal{Q}}^\beta \mathcal{E}$, we denote the corresponding structural projections as $\mathrm{pr}_1^{\beta}(\Pi) = P$ and $\mathrm{pr}_2^{\beta}(\Pi) = U$. In local adapted coordinates, if $P = (q,\kappa,p_q,p_{\kappa}) \in T^*T^*\mathcal{Q}$ and $U = (q,v_q = p_{\kappa},u)$, then, we label the corresponding point in $TT\mathcal{Q} \oplus_{T\mathcal{Q}}^\kappa \mathcal{E}$ by $Y = (q,\kappa,p_q,p_{\kappa},u)$. Thus, $\beta^{-1}_\mathcal{Q}(q,\kappa,p_q,p_{\kappa}) = (q,\kappa,p_{\kappa},-p_{q})$ and $T\pi_{Q}(q,\kappa,p_{\kappa},-p_{q}) = (q,v_q = p_{\kappa}) = \pi^{\mathcal{E}}(q,v_q,u)$.\\

Similar to the $\alpha_{\mathcal{Q}}$-twisted case, this allows us to define the diffeomorphism $\beta_\mathcal{Q}^{\mathcal{E}}: TT^*\mathcal{Q} \oplus_{T\mathcal{Q}}^\alpha \mathcal{E} \to T^*T^*\mathcal{Q} \oplus_{T\mathcal{Q}}^{\beta} \mathcal{E}$, the analogue of $\beta_\mathcal{Q}$ extended to the sum of vector bundles, that makes the following diagram commutative.

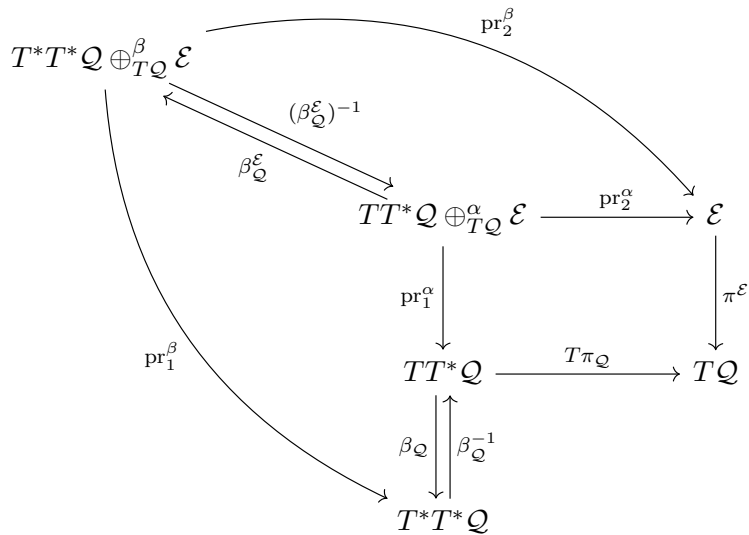
\begin{center}
    \begin{tikzcd}
    T^*T^*\mathcal{Q} \oplus_{T\mathcal{Q}}^{\beta} \mathcal{E} \arrow[rrrrdd, "\mathrm{pr}^{\beta}_2", bend left] \arrow[rrdddddd, "\mathrm{pr}^{\beta}_1"', bend right] \arrow[rrdd, "(\beta_{\mathcal{Q}}^{\mathcal{E}})^{-1}", shift left] &  &                                                                                                                                                                                                       &  &                                             \\
                                                                                                                                                                                                                                               &  &                                                                                                                                                                                                       &  &                                             \\
                                                                                                                                                                                                                                               &  & TT^*{\mathcal{Q}} \oplus_{T\mathcal{Q}}^{\alpha} \mathcal{E} \arrow[rr, "\mathrm{pr}_2^{\alpha}"] \arrow[lluu, "\beta_{\mathcal{Q}}^{\mathcal{E}}", shift left] \arrow[dd, "\mathrm{pr}_1^{\alpha}"'] &  & \mathcal{E} \arrow[dd, "\pi^{\mathcal{E}}"] \\
                                                                                                                                                                                                                                               &  &                                                                                                                                                                                                       &  &                                             \\
                                                                                                                                                                                                                                               &  & TT^*\mathcal{Q} \arrow[dd, "\beta_\mathcal{Q}"', shift right] \arrow[rr, "T\pi_{\mathcal{Q}}"]                                                                                                        &  & T\mathcal{Q}                                \\
                                                                                                                                                                                                                                               &  &                                                                                                                                                                                                       &  &                                             \\
                                                                                                                                                                                                                                               &  & T^*T^*\mathcal{Q} \arrow[uu, "\beta_{\mathcal{Q}}^{-1}"', shift right]                                                                                                                                &  &                                            
    \end{tikzcd}
\end{center}

With these, we have managed to extend Tulczyjew's triple to the sum of vector bundles that naturally occurs in the case of optimal control of second-order systems. The situation can be summarized in the following diagram.

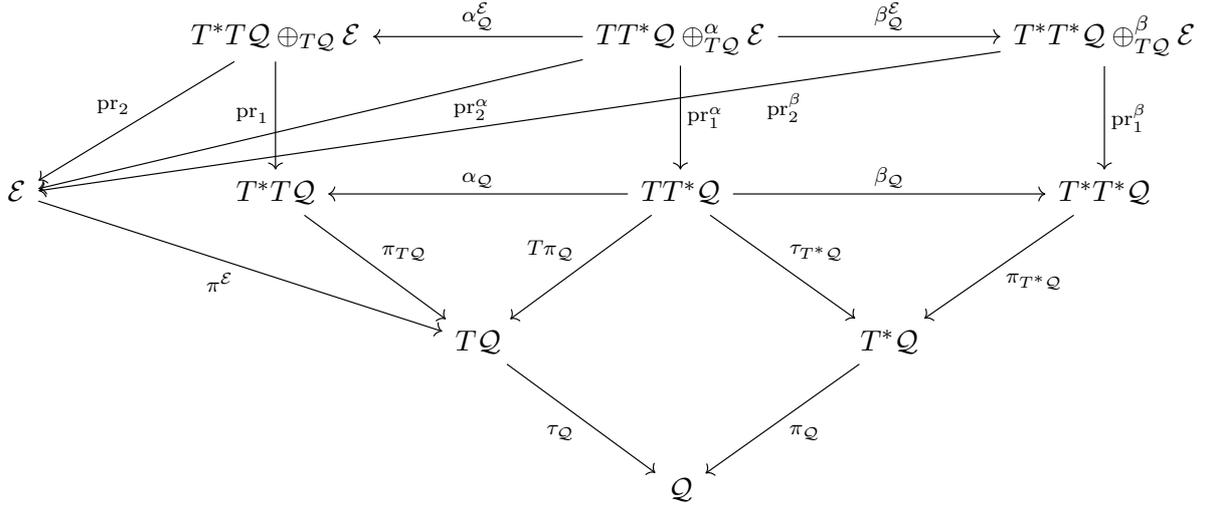
\begin{center}
    \begin{tikzcd}
                                                    &  & T^*T\mathcal{Q} \oplus_{T\mathcal{Q}} \mathcal{E} \arrow[dd, "\mathrm{pr}_1"'] \arrow[lldd, "\mathrm{pr}_2"'] &                                                 & TT^*\mathcal{Q} \oplus_{T\mathcal{Q}}^{\alpha} \mathcal{E} \arrow[dd, "\mathrm{pr}_1^{\alpha}"] \arrow[ll, "\alpha_{\mathcal{Q}}^{\mathcal{E}}"'] \arrow[rr, "\beta_{\mathcal{Q}}^{\mathcal{E}}"] \arrow[lllldd, "\mathrm{pr}_2^{\alpha}" near start] &                                                 & T^*T^*\mathcal{Q} \oplus_{T\mathcal{Q}}^{\beta} \mathcal{E} \arrow[dd, "\mathrm{pr}_1^{\beta}"] \arrow[lllllldd, "\mathrm{pr}_2^{\beta}" near start] \\
                                                    &  &                                                                                                               &                                                 &                                                                                                                                                                                                                                            &                                                 &                                                                                                                                           \\
    \mathcal{E} \arrow[rrrdd, "\pi^{\mathcal{E}}"'] &  & T^*T\mathcal{Q} \arrow[rdd, "\pi_{T\mathcal{Q}}"]                                                             &                                                 & TT^*\mathcal{Q} \arrow[ldd, "T\pi_{\mathcal{Q}}"'] \arrow[rdd, "\tau_{T^*\mathcal{Q}}"] \arrow[ll, "\alpha_{\mathcal{Q}}"'] \arrow[rr, "\beta_{\mathcal{Q}}"]                                                                              &                                                 & T^*T^*\mathcal{Q} \arrow[ldd, "\pi_{T^*\mathcal{Q}}"]                                                                                     \\
                                                    &  &                                                                                                               &                                                 &                                                                                                                                                                                                                                            &                                                 &                                                                                                                                           \\
                                                    &  &                                                                                                               & T\mathcal{Q} \arrow[rdd, "\tau_{\mathcal{Q}}"'] &                                                                                                                                                                                                                                            & T^*\mathcal{Q} \arrow[ldd, "\pi_{\mathcal{Q}}"] &                                                                                                                                           \\
                                                    &  &                                                                                                               &                                                 &                                                                                                                                                                                                                                            &                                                 &                                                                                                                                           \\
                                                    &  &                                                                                                               &                                                 & \mathcal{Q}                                                                                                                                                                                                                                &                                                 &                                                                                                                                          
    \end{tikzcd}
\end{center}

\begin{remark}
    In the case mentioned in Remark~\ref{rmk:controlled_SODE_lower_control_space}, where $\pi^{\check{\mathcal{E}}} : \check{\mathcal{E}} \to \mathcal{Q}$ exists, then, the previous constructions become unnecessary and Tulcyjew's diffeomorphisms extend trivially to $(T^*T\mathcal{Q} \oplus_\mathcal{Q} \check{\mathcal{E}},TT^*\mathcal{Q} \oplus_\mathcal{Q} \check{\mathcal{E}},T^*T^*\mathcal{Q} \oplus_\mathcal{Q} \check{\mathcal{E}})$.
\end{remark}

\subsection{Force-controlled Lagrangian and Hamiltonian systems}
Turning our attention to the reformulation in terms of force-controlled Lagrangian systems, we begin by providing the following

\begin{definition}
    Let
    \begin{equation*}
        TT\mathcal{Q} \oplus_{T\mathcal{Q}}^\kappa \mathcal{E} = \left\lbrace (X,U) \in TT\mathcal{Q} \times \mathcal{E} \;\vert\; \tau_{T\mathcal{Q}} \circ \kappa_{\mathcal{Q}} (X) = T\tau_{\mathcal{Q}} (X) = \pi^{\mathcal{E}}(U) \right\rbrace
    \end{equation*}
    We say $TT\mathcal{Q} \oplus_{T\mathcal{Q}}^\kappa \mathcal{E}$ is the $\kappa_\mathcal{Q}$-twisted sum of $TT\mathcal{Q}$ and $\mathcal{E}$.
\end{definition}

Given $Y \in TT\mathcal{Q} \oplus_{T\mathcal{Q}}^\kappa \mathcal{E}$, we denote the corresponding structural projections as $\mathrm{pr}_1^{\kappa}(Y) = X$ and $\mathrm{pr}_2^{\kappa}(Y) = U$. In local adapted coordinates, if $X = (q,\xi,v_q,v_{\xi}) \in TT\mathcal{Q}$ and $U = (q,v_q,u)$, then, we label the corresponding point in $TT\mathcal{Q} \oplus_{T\mathcal{Q}}^\kappa \mathcal{E}$ by $Y = (q,\xi,v_q,v_{\xi},u)$. Thus, $\kappa_\mathcal{Q}(q,\xi,v_q,v_{\xi}) = (q,v_q,\xi,v_{\xi})$ and $\tau_{TQ}(q,v_q,v_{\kappa},\kappa) = (q,v_q) = \pi^{\mathcal{E}}(q,v_q,u)$.\\

Similar to the $\alpha_{\mathcal{Q}}$-twisted case, this allows us to define the diffeomorphism $\kappa_\mathcal{Q}^{\mathcal{E}}: TT\mathcal{Q} \oplus_{T\mathcal{Q}}^\kappa \mathcal{E} \to TT\mathcal{Q} \oplus_{T\mathcal{Q}} \mathcal{E}$, the analogue of $\kappa_\mathcal{Q}$ extended to the sum of vector bundles, that makes the following diagram commutative. In contrast with $\kappa_\mathcal{Q}$, it is no longer an involution.


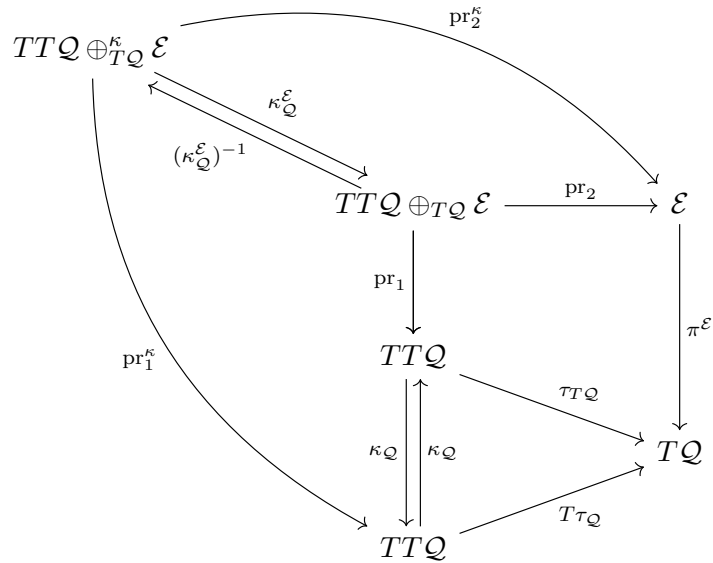
\begin{center}
        \begin{tikzcd}
        TT\mathcal{Q} \oplus_{T\mathcal{Q}}^{\kappa} \mathcal{E} \arrow[rrrrdd, "\mathrm{pr}^{\kappa}_2", bend left] \arrow[rrdddddd, "\mathrm{pr}^{\kappa}_1"', bend right] \arrow[rrdd, "\kappa_{\mathcal{Q}}^\mathcal{E}", shift left] &  &                                                                                                                                                                                             &  &                                              \\
                                                                                                                                                                                                                                          &  &                                                                                                                                                                                             &  &                                              \\
                                                                                                                                                                                                                                          &  & TT{\mathcal{Q}} \oplus_{T\mathcal{Q}} \mathcal{E} \arrow[rr, "\mathrm{pr}_2"] \arrow[dd, "\mathrm{pr}_1"'] \arrow[dd] \arrow[lluu, "(\kappa_{\mathcal{Q}}^{\mathcal{E}})^{-1}", shift left] &  & \mathcal{E} \arrow[ddd, "\pi^{\mathcal{E}}"] \\
                                                                                                                                                                                                                                          &  &                                                                                                                                                                                             &  &                                              \\
                                                                                                                                                                                                                                          &  & TT\mathcal{Q} \arrow[dd, "\kappa_\mathcal{Q}"', shift right] \arrow[rrd, "\tau_{T\mathcal{Q}}"]                                                                                             &  &                                              \\
                                                                                                                                                                                                                                          &  &                                                                                                                                                                                             &  & T\mathcal{Q}                                 \\
                                                                                                                                                                                                                                          &  & TT\mathcal{Q} \arrow[rru, "T\tau_{\mathcal{Q}}"'] \arrow[uu, "\kappa_{\mathcal{Q}}"', shift right]                                                                                          &  &                                             
        \end{tikzcd}
\end{center}


With this structure, the new control Lagrangian for the force-controlled Lagrangian system $(\mathcal{Q},\mathcal{E},L,f_L^{\mathcal{E}})$, can be rewritten as $\tilde{\mathcal{L}}_L^{\mathcal{E}} = \mathcal{L}_L^{\mathcal{E}} \circ \kappa_{\mathcal{Q}}^{\mathcal{E}}$, where $\mathcal{L}_L^\mathcal{E} = \left\langle \left( dL \circ \pi^{\mathcal{E}} + f_L^{\mathcal{E}} \right) \circ \mathrm{pr}_2 , \mathrm{pr}_1  \right\rangle_{T\mathcal{Q}} - C \circ \mathrm{pr}_2$ with $dL$ the exterior derivative of $L$, or, in local coordinates,
        \begin{equation*}
            \mathcal{L}^\mathcal{E}_L(q,v_q,\xi,v_{\xi},u) = \tilde{\mathcal{L}}_L^{\mathcal{E}}(q,\xi,v_q,v_{\xi},u)\,.
        \end{equation*}

The presence of a hyperregular force-controlled Lagrangian system induces isomorphisms
\begin{align*}
    \sharp_{L,f_L^{\mathcal{E}}} &: T^*T\mathcal{Q} \oplus_{T\mathcal{Q}} \mathcal{E} \to TT\mathcal{Q} \oplus_{T\mathcal{Q}} \mathcal{E}\\
    \flat_{L,f_L^{\mathcal{E}}}&: TT\mathcal{Q} \oplus_{T\mathcal{Q}} \mathcal{E} \to T^*T\mathcal{Q} \oplus_{T\mathcal{Q}} \mathcal{E}
\end{align*}
whose local coordinate form is
\begin{align*}
    &\sharp_{L,f_L^{\mathcal{E}}}( q, v_q, v_{\kappa}, \kappa, u)\\
    &\quad = ( q, v_q, \xi = \kappa (D_{2 2} L(q,v_q))^{-1}, v_{\xi} = v_{\kappa} (D_{2 2} L(q,v_q))^{-1} - \kappa (D_{2 2} L(q,v_q))^{-1} D_2 f_{L}^{\mathcal{E}}(q,v_q,u), u)\\
   &\flat_{L,f_L^{\mathcal{E}}}( q, v_q, \xi, v_{\xi}, u) \\
   &\quad = ( q, v_q, v_{\kappa} = v_{\xi} D_{2 2} L(q,v_q) + \xi D_2 f_{L}^{\mathcal{E}}(q,v_q,u), \kappa = \xi D_{2 2}L(q,v_q), u)
\end{align*}

With all of this, one can check that
\begin{align*}
    \tilde{\mathcal{L}}^{\mathcal{E}} \circ (\alpha_{\mathcal{Q}}^{\mathcal{E}})^{-1} \circ \flat_{L,f_L^{\mathcal{E}}} \circ \kappa_{\mathcal{Q}}^{\mathcal{E}} &= \tilde{\mathcal{L}}^{\mathcal{E}}_L + \text{total derivative,}\\
    \tilde{\mathcal{L}}_L^{\mathcal{E}} \circ (\kappa_{\mathcal{Q}}^{\mathcal{E}})^{-1} \circ \sharp_{L,f_L^{\mathcal{E}}} \circ \alpha_{\mathcal{Q}}^{\mathcal{E}} &= \tilde{\mathcal{L}}^{\mathcal{E}} + \text{total derivative,}
\end{align*}
meaning that, as expected, given a hyperregular force-controlled Lagrangian system, working with either $\tilde{\mathcal{L}}^{\mathcal{E}}$ and $\tilde{\mathcal{L}}^{\mathcal{E}}_L$ is equivalent. The total derivatives that appear provide suitable transformations of the new boundary costs that appear both in $\tilde{J}_3$ and $\tilde{J}_4$ so that the augmented objective functions coincide.\\

When moving to the new Hamiltonian picture for force-controlled Lagrangian systems, we need to define the following space.

\begin{definition}
    Let
    \begin{equation*}
        T^*T\mathcal{Q} \oplus_{T\mathcal{Q}}^{\alpha,L} \mathcal{E} = \left\lbrace (\Lambda,U) \in T^*T\mathcal{Q} \times \mathcal{E} \;\vert\; \mathbb{F}L^{-1} \circ \pi_{T\mathcal{Q}} \circ \alpha^{-1}_{\mathcal{Q}} (\Lambda) = \pi^{\mathcal{E}}(U) \right\rbrace
    \end{equation*}
    We say $TT^*\mathcal{Q} \oplus_{T\mathcal{Q}}^{\alpha,L} \mathcal{E}$ is the $(\alpha_\mathcal{Q},L)$-twisted sum of $T^*T\mathcal{Q}$ and $\mathcal{E}$.
\end{definition}

Given $\mathrm{M} \in T^*T\mathcal{Q} \oplus_{T\mathcal{Q}}^{\alpha,L} \mathcal{E}$, we denote the corresponding structural projections as $\mathrm{pr}_1^{\alpha, L}(\mathrm{M}) = \Lambda$ and $\mathrm{pr}_2^{\alpha, L}(\mathrm{M}) = U$. In local adapted coordinates, if $\Lambda = (q,\xi,\varpi_q,\varpi_{\xi}) \in T^*T\mathcal{Q}$ and $U = (q,v_q,u)$, with $\varpi_{\xi}^{\top} = D_2 L(q,v_q)$, then, we may label the corresponding point in $T^*T\mathcal{Q} \oplus_{T\mathcal{Q}}^{\alpha,L} \mathcal{E}$ by $\mathrm{M} = (q,\xi,\varpi_q,\varpi_{\xi},u)$. Clearly, given the dependence of this definition on $L$, this space is not canonical.\\

In Section \ref{ssec:OptimalControlLagrangian} we also considered the case where our dynamics was given by a force-controlled Hamiltonian system with $H: T^*\mathcal{Q} \to \mathbb{R}$. There, the control space was assumed to be $(\mathcal{F},\pi^{\mathcal{F}},T^*\mathcal{Q})$. Tulczyjew's triple extends similarly to the corresponding sum of vector bundles. In particular, since the control fibers are over $T^*\mathcal{Q}$ we can work directly with $TT^*\mathcal{Q} \oplus_{T^*\mathcal{Q}} \mathcal{F}$ and $T^*T^*\mathcal{Q} \oplus_{T^*\mathcal{Q}} \mathcal{F}$. These can be related by extending $\beta_{\mathcal{Q}}$ trivially to the sum requiring the following diagram to commute.

\begin{center}
    \begin{tikzcd}
    TT^*\mathcal{Q} \oplus_{T^*\mathcal{Q}} \mathcal{F} \arrow[dd, "\mathrm{pr}_1"'] \arrow[rr, "\tilde{\beta}_{\mathcal{Q}}^{\mathcal{F}}"] \arrow[rd, "\mathrm{pr}_2"] &             & T^*T^*\mathcal{Q} \oplus_{T^*\mathcal{Q}} \mathcal{F} \arrow[dd, "\mathrm{pr}_1"] \arrow[ld, "\mathrm{pr}_2"'] \\
                                                                                                                                                                         & \mathcal{F} &                                                                                                                \\
    TT^*\mathcal{Q} \arrow[rr, "\beta_{\mathcal{Q}}"]                                                                                                                    &             & T^*T^*\mathcal{Q}                                                                                             
    \end{tikzcd}
\end{center}

Thus, in this case, we only need the following
\begin{definition}
    Let
    \begin{equation*}
        T^*T\mathcal{Q} \oplus_{T^*\mathcal{Q}}^{\tilde{\alpha}} \mathcal{F} = \left\lbrace (\mathrm{H},W) \in T^*T\mathcal{Q} \times \mathcal{F} \;\vert\; \tau_{T^*\mathcal{Q}} \circ \alpha^{-1}_{\mathcal{Q}} (\mathrm{H}) = \pi^{\mathcal{F}}(W) \right\rbrace
    \end{equation*}
    We say $T^*T\mathcal{Q} \oplus_{T^*\mathcal{Q}}^{\tilde{\alpha}} \mathcal{F}$ is the $\alpha_{\mathcal{Q}}$-twisted sum of $T^*T\mathcal{Q}$ and $\mathcal{F}$.
\end{definition}

Given $\Xi \in T^*T\mathcal{Q} \oplus_{T^*\mathcal{Q}}^{\tilde{\alpha}} \mathcal{F}$, we denote the corresponding structural projections as $\mathrm{pr}_1^{\tilde{\alpha}}(\Xi) = \mathrm{H}$ and $\mathrm{pr}_2^{\tilde{\alpha}}(\Xi) = W$. 
In local adapted coordinates, if $\mathrm{H} = (q,\xi,\varpi_{q},\varpi_{\xi}) \in T^*T\mathcal{Q}$ and $W = (q,\varpi_{\xi},w)$, then, we label the corresponding point in $T^*T\mathcal{Q} \oplus_{T^*\mathcal{Q}}^{\tilde{\alpha}} \mathcal{F}$ by $\Xi = (q,\xi,\varpi_{q},\varpi_{\xi},w)$.\\

This structure allows us to define the diffeomorphism $\tilde{\alpha}_\mathcal{Q}^{\mathcal{F}}: TT^*\mathcal{Q} \oplus_{T^*\mathcal{Q}} \mathcal{F} \to T^*T\mathcal{Q} \oplus_{T^*\mathcal{Q}}^{\tilde{\alpha}} \mathcal{F}$, that makes the following diagram commute.

\begin{center}
    \begin{tikzcd}
    T^*T\mathcal{Q} \oplus_{T^*\mathcal{Q}}^{\tilde{\alpha}} \mathcal{F} \arrow[rrrrdd, "\mathrm{pr}^{\tilde{\alpha}}_2", bend left] \arrow[rrdddddd, "\mathrm{pr}^{\tilde{\alpha}}_1"', bend right] \arrow[rrdd, "(\tilde{\alpha}_{\mathcal{Q}}^{\mathcal{F}})^{-1}", shift left] &  &                                                                                                                                                                                                &  &                                             \\
                                                                                                                                                                                                                                                                                   &  &                                                                                                                                                                                                &  &                                             \\
                                                                                                                                                                                                                                                                                   &  & TT^*{\mathcal{Q}} \oplus_{T^*\mathcal{Q}} \mathcal{F} \arrow[rr, "\mathrm{pr}_2"] \arrow[dd, "\mathrm{pr}_1"'] \arrow[dd] \arrow[lluu, "\tilde{\alpha}_{\mathcal{Q}}^\mathcal{F}", shift left] &  & \mathcal{F} \arrow[dd, "\pi^{\mathcal{F}}"] \\
                                                                                                                                                                                                                                                                                   &  &                                                                                                                                                                                                &  &                                             \\
                                                                                                                                                                                                                                                                                   &  & TT^*\mathcal{Q} \arrow[dd, "\alpha_{\mathcal{Q}}"', shift right] \arrow[rr, "\tau_{T^*\mathcal{Q}}"]                                                                                           &  & T^*\mathcal{Q}                              \\
                                                                                                                                                                                                                                                                                   &  &                                                                                                                                                                                                &  &                                             \\
                                                                                                                                                                                                                                                                                   &  & TT^*\mathcal{Q} \arrow[uu, "\alpha_{\mathcal{Q}}^{-1}"', shift right]                                                                                                                          &  &                                            
    \end{tikzcd}
\end{center}

In Section \ref{ssec:OptimalControlLagrangian} we already mentioned the possibility of having a vector bundle isomorphism over $\mathbb{F}L$, $\chi_{\mathbb{F}L}: \mathcal{E} \to \mathcal{F}$, relating both control bundles. One can then define a sum of vector bundles extension, $\chi_{\mathbb{F}L}^{\oplus}: T^*T\mathcal{Q} \oplus_{T\mathcal{Q}}^{\alpha,L} \mathcal{E} \to T^*T\mathcal{Q} \oplus_{T^*\mathcal{Q}}^{\tilde{\alpha}} \mathcal{F}$, by requiring that the following diagram commutes.
\begin{center}
    \begin{tikzcd}
    {T^*T\mathcal{Q} \oplus_{T\mathcal{Q}}^{\alpha,L} \mathcal{E}} \arrow[dd, "{\mathrm{pr}_2^{\alpha,L}}"'] \arrow[rr, "\chi_{\mathbb{F}L}^{\oplus}"] \arrow[rd, "{\mathrm{pr}_1^{\alpha,L}}"] &                 & T^*T\mathcal{Q} \oplus_{T^*\mathcal{Q}}^{\tilde{\alpha}} \mathcal{F} \arrow[dd, "\mathrm{pr}_{2}^{\tilde{\alpha}}"] \arrow[ld, "\mathrm{pr}_1^{\tilde{\alpha}}"'] \\
                                                                                                                                                                                                & T^*T\mathcal{Q} &                                                                                                                                                                   \\
    \mathcal{E} \arrow[rr, "\chi_{\mathbb{F}L}"]                                                                                                                                                &                 & \mathcal{F}                                                                                                                                                      
    \end{tikzcd}
\end{center}

All of this provides us with the following extended Tulczyjew's triple in the case of optimal control of force-controlled Lagrangian systems.

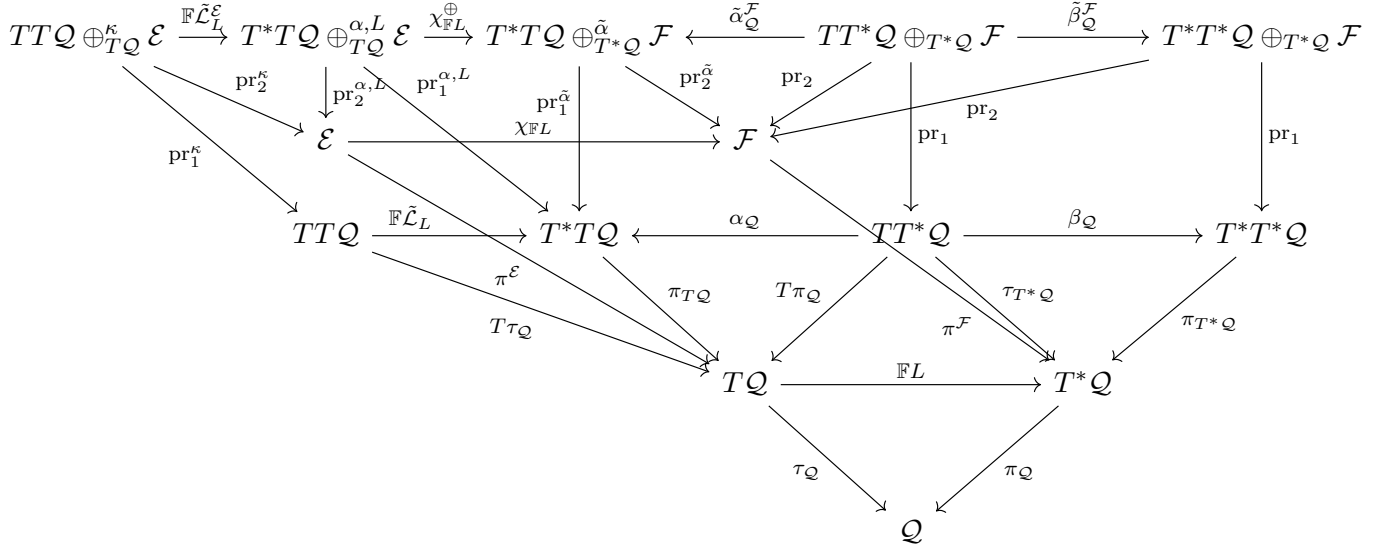
\begin{center}
    \begin{tikzcd}[column sep=0.85em]
    TT\mathcal{Q} \oplus_{T\mathcal{Q}}^{\kappa} \mathcal{E} \arrow[rr, "\mathbb{F} \tilde{\mathcal{L}}_L^{\mathcal{E}}"] \arrow[rrd, "\mathrm{pr}_2^{\kappa}"] \arrow[rrdd, "\mathrm{pr}_1^{\kappa}"'] &  & {T^*T\mathcal{Q} \oplus_{T\mathcal{Q}}^{\alpha,L} \mathcal{E}} \arrow[rr, "\chi_{\mathbb{F}L}^{\oplus}"] \arrow[d, "{\mathrm{pr}_2^{\alpha,L}}"] \arrow[rrdd, "{\mathrm{pr}_1^{\alpha,L}}" near start] &  & T^*T\mathcal{Q} \oplus_{T^*\mathcal{Q}}^{\tilde{\alpha}} \mathcal{F} \arrow[dd, "\mathrm{pr}_1^{\tilde{\alpha}}"' near start] \arrow[rd, "\mathrm{pr}_2^{\tilde{\alpha}}"] &                                                                           & TT^*\mathcal{Q} \oplus_{T^*\mathcal{Q}} \mathcal{F} \arrow[dd, "\mathrm{pr}_1"] \arrow[ll, "\tilde{\alpha}_{\mathcal{Q}}^{\mathcal{F}}"'] \arrow[rr, "\tilde{\beta}_{\mathcal{Q}}^{\mathcal{F}}"] \arrow[ld, "\mathrm{pr}_2"'] &                                                 & T^*T^*\mathcal{Q} \oplus_{T^*\mathcal{Q}} \mathcal{F} \arrow[dd, "\mathrm{pr}_1"] \arrow[llld, "\mathrm{pr}_2"] \\
                                                                                                                                                                                                        &  & \mathcal{E} \arrow[rrr, "\chi_{\mathbb{F}L}"] \arrow[rrrddd, "\pi^{\mathcal{E}}"']                                                                                                           &  &                                                                                                                                                                 & \mathcal{F} \arrow[rrddd, "\pi^\mathcal{F}"' near end]                             &                                                                                                                                                                                                                                &                                                 &                                                                                                                 \\
                                                                                                                                                                                                        &  & TT\mathcal{Q} \arrow[rrrdd, "T\tau_{\mathcal{Q}}"'] \arrow[rr, "\mathbb{F}\tilde{\mathcal{L}}_L" near start]                                                                                            &  & T^*T\mathcal{Q} \arrow[rdd, "\pi_{T\mathcal{Q}}"]                                                                                                               &                                                                           & TT^*\mathcal{Q} \arrow[ldd, "T\pi_{\mathcal{Q}}"'] \arrow[rdd, "\tau_{T^*\mathcal{Q}}"] \arrow[ll, "\alpha_{\mathcal{Q}}"'] \arrow[rr, "\beta_{\mathcal{Q}}"]                                                                  &                                                 & T^*T^*\mathcal{Q} \arrow[ldd, "\pi_{T^*\mathcal{Q}}"]                                                           \\
                                                                                                                                                                                                        &  &                                                                                                                                                                                              &  &                                                                                                                                                                 &                                                                           &                                                                                                                                                                                                                                &                                                 &                                                                                                                 \\
                                                                                                                                                                                                        &  &                                                                                                                                                                                              &  &                                                                                                                                                                 & T\mathcal{Q} \arrow[rdd, "\tau_{\mathcal{Q}}"'] \arrow[rr, "\mathbb{F}L"] &                                                                                                                                                                                                                                & T^*\mathcal{Q} \arrow[ldd, "\pi_{\mathcal{Q}}"] &                                                                                                                 \\
                                                                                                                                                                                                        &  &                                                                                                                                                                                              &  &                                                                                                                                                                 &                                                                           &                                                                                                                                                                                                                                &                                                 &                                                                                                                 \\
                                                                                                                                                                                                        &  &                                                                                                                                                                                              &  &                                                                                                                                                                 &                                                                           & \mathcal{Q}                                                                                                                                                                                                                    &                                                 &                                                                                                                
    \end{tikzcd}
\end{center}

\subsection{Relationship with Pontryagin's Hamiltonian}
With these definitions at hand, we can state the following
\begin{theorem}
    \label{thm:relation_hamiltonian_new_functions}
    Consider a control Hamiltonian $\mathcal{H}_{\lambda_0}: T^*T\mathcal{Q} \oplus_{T\mathcal{Q}} \mathcal{E} \to \mathbb{R}$ (see Eq.~\eqref{eq:Pontryagins_control_Hamiltonian}) for a controlled SODE $X$ with runnning cost $C: \mathcal{E} \to \mathbb{R}$. Then,
    \begin{align*}
        \tilde{\mathcal{L}}^\mathcal{E} &= \hphantom{-}\mathcal{H}_{-1} \circ \alpha_{\mathcal{Q}}^{\mathcal{E}}\,,\\
        \tilde{\mathcal{H}}^\mathcal{E} &= -\mathcal{H}_{-1} \circ \alpha_{\mathcal{Q}}^{\mathcal{E}} \circ (\beta_{\mathcal{Q}}^{\mathcal{E}})^{-1}\,.
    \end{align*}
\end{theorem}

\begin{proof}
    It suffices to check these locally in an adapted coordinate system.
\end{proof}

\begin{remark}
    These suggest that we may as well extend our definitions of $\tilde{\mathcal{L}}^\mathcal{E}$ and $\tilde{\mathcal{H}}^\mathcal{E}$ to arbitrary $\lambda_0$.
\end{remark}

We have omitted up until this point the fact that a big part of the importance of Tulczyjew's triple in geometric mechanics comes from the fact that each of the double bundles involved is a symplectic manifold, namely $(T^*T\mathcal{Q}, \omega_{T\mathcal{Q}})$ and $(T^*T^*\mathcal{Q}, \omega_{T^*\mathcal{Q}})$ are naturally symplectic. Moreover, $(TT^*\mathcal{Q}, \omega_{\alpha} = -\omega_{\beta})$ is also a symplectic manifold with $\omega_{\alpha} = \alpha_{\mathcal{Q}}^* \, \omega_{T\mathcal{Q}}$ and $\omega_{\beta} = \beta_{\mathcal{Q}}^* \, \omega_{T^*\mathcal{Q}}$, making $\alpha_{\mathcal{Q}}$ and $\beta_{\mathcal{Q}}$ symplectomorphisms (the latter actually an anti-symplectomorphism).\\

In our extended setting, what we find is that $(T^*T\mathcal{Q} \oplus_{T\mathcal{Q}} \mathcal{E}, \omega_{T\mathcal{Q}}^{\mathcal{E}} = \mathrm{pr}_1^*\, \omega_{T\mathcal{Q}})$ and $(T^*T^*\mathcal{Q} \oplus_{T\mathcal{Q}}^{\beta} \mathcal{E}, \omega_{T^*\mathcal{Q}}^{\mathcal{E}} = (\mathrm{pr}_1^{\beta})^* \, \omega_{T^*\mathcal{Q}})$ are presymplectic manifolds and $(TT^*\mathcal{Q} \oplus_{T\mathcal{Q}}^{\alpha} \mathcal{E}, \omega_{\alpha}^{\mathcal{E}} = -\omega_{\beta}^{\mathcal{E}})$ with $\omega_{\alpha}^{\mathcal{E}} = (\alpha_{\mathcal{Q}}^{\mathcal{E}})^* \, \omega_{T\mathcal{Q}}^{\mathcal{E}}$ and $\omega_{\beta}^{\mathcal{E}} = (\beta_{\mathcal{Q}}^{\mathcal{E}})^* \, \omega_{T^*\mathcal{Q}}^{\mathcal{E}}$ is also presymplectic. Locally, with our choices of coordinates we have that
\begin{align*}
	\omega_{T\mathcal{Q}}^\mathcal{E} &= d q^i \wedge d \lambda_{q,i} + d v^i \wedge d \lambda_{v,i}\,,\\
	\omega_{T^*\mathcal{Q}}^\mathcal{E} &= d q^i \wedge d p_{q,i} + d \kappa_i \wedge d p_{\kappa}^i\,,\\
	\omega_{\alpha}^\mathcal{E} &= d q^i  \wedge d v_{\kappa,i} + d v_q^i \wedge d\kappa _i\,.
\end{align*}

The necessary conditions for optimality (modulo transversality conditions) stemming from \eqref{eq:generic_augmented_objective} can be recast in the following compact geometric form:
\begin{equation*}
\imath_{X_{\mathcal{H}}^\mathcal{E}} \omega_{T\mathcal{Q}}^\mathcal{E} = d\mathcal{H}_{-1}\,.
\end{equation*}

As a direct result from Theorem~\ref{thm:relation_hamiltonian_new_functions}, if $X_{\alpha}^\mathcal{E} =(\alpha_{\mathcal{Q}}^\mathcal{E})_*^{-1} X_{\mathcal{H}}^\mathcal{E}$ and $X_{\hat{\mathcal{H}}}^\mathcal{E} = (\beta_{\mathcal{Q}}^\mathcal{E})_* (\alpha_{\mathcal{Q}}^\mathcal{E})_*^{-1} X_{\mathcal{H}}^\mathcal{E}$, then, 
\begin{align*}
    \imath_{X_{\alpha}^\mathcal{E}} \omega_{\alpha}^{\mathcal{E}} &= d\tilde{\mathcal{L}}^\mathcal{E}\,,\\
	\imath_{X_{\tilde{\mathcal{H}}}^\mathcal{E}} \omega_{T^*\mathcal{Q}}^\mathcal{E} &= d\tilde{\mathcal{H}}^\mathcal{E}\,.
\end{align*}

The first equation is quite surprising, particularly in light of \eqref{eq:EL_geometric}. Indeed, $X_{\alpha}^{\mathcal{E}}$ is not vector field generated by the Euler-Lagrange equations derived from $\tilde{J}_3$. In order to actually regain those, one needs to proceed to construct the Poincar{\'e}-Cartan 2-form either from the canonical structure of a tangent bundle or, equivalently, by pullback of $\omega_{T^*\mathcal{Q}}^{\mathcal{E}}$ through the fiber derivative, following the same procedure as in Section~\ref{sssec:force_controlled EL_eqs}, leading to $\omega_{\tilde{\mathcal{L}}^{\mathcal{E}}} = (\mathbb{F}\tilde{L}^{\mathcal{E}})^* \omega_{T^*\mathcal{Q}}^{\mathcal{E}}$. By the hyperregularity of the new Lagrangian, it is yet another presymplectic form in $TT^*\mathcal{Q} \oplus_{T\mathcal{Q}}^{\alpha} \mathcal{E}$. Actually,
\begin{equation*}
    \omega_{\tilde{\mathcal{L}}^{\mathcal{E}}} = \omega_{\alpha}^{\mathcal{E}} + \zeta
\end{equation*}
where $\zeta$ is another exact form. Then, the necessary conditions for optimality for $\tilde{J}_3$ can be rewritten as
\begin{equation*}
    \imath_{X_{\tilde{\mathcal{L}}}^\mathcal{E}} \omega_{\tilde{\mathcal{L}}^{\mathcal{E}}} = dE_{\tilde{\mathcal{L}}^\mathcal{E}}\,.
\end{equation*}
Here, the vector field $X_{\tilde{\mathcal{L}}}^\mathcal{E} \neq  X_{\alpha}^{\mathcal{E}}$ is the one that corresponds to the Euler-Lagrange equations. In local adapted coordinates, we have
\begin{align*}
    X_{\alpha}^{\mathcal{E}} &= v_q \partial_{q}
    + \left( D_2 C(q, v_q, u) - v_{\kappa}^{\top} - \kappa^{\top} D_2 f(q, v_q, u) \right) \partial_{\kappa}\\
    &+ f(q, v_q, u) \partial_{v_{q}}
    + \left( D_1 C(q, v_q, u) - \kappa^{\top} D_1 f(q, v_q, u)\right) \partial_{v_{\kappa}} + X_u \partial_u\,,\\
    X_{\tilde{\mathcal{L}}}^{\mathcal{E}} &= v_q \partial_{q}
    + v_{\kappa} \partial_{\kappa} + f(q, v_q, u) \partial_{v_{q}}\\
    &+ \left[ \left( D_{21} C - \kappa^{\top} D_{2 1} f \right) v_q + \left( D_{22} C - \kappa^{\top} D_{2 2} f \right) f + \left( D_{23} C - \kappa^{\top} D_{2 3} f \right) X_u - \left( D_1 C - \kappa^{\top} D_1 f\right)  \right] \partial_{v_{\kappa}}\\
    &+ X_u \partial_u\,.
\end{align*}
In the former, $v_{\kappa}$ does not play the role of the time derivative of $\kappa$. It is simply a fiber coordinate over $\kappa$. This can be understood once we realize that $v_{\kappa}$ is the image of $\lambda_q$ in $TT^*\mathcal{Q}$. In the latter, $v_{\kappa}$ does indeed play the role of derivative of $\kappa$, with $X_{\tilde{\mathcal{L}}}^{\mathcal{E}}$ being a controlled SODE.\\

Whenever $\bar{u}$ is an optimal control, our sum of vector bundles collapse into the original Tulczyjew's triple and the resulting vector fields associated to the corresponding optimal control Hamiltonian, optimal new Lagrangian and optimal new Hamiltonian, $X_{\bar{\mathcal{H}}}$ (with $\bar{\mathcal{H}}_{-1}$ denoting the optimal Pontryagin's control Hamiltonian), $X_{\tilde{\mathcal{L}}}$ and $X_{\tilde{\mathcal{H}}}$ respectively, become Hamiltonian vector fields and their flows conserve the respective symplectic forms.

\begin{remark}
    Tulczyjew's triple in mechanics provides an invariant way to understand the relation between Lagrangian and Hamiltonian mechanics. More precisely, given hyperregular $L$ and $H$, $\mathrm{im}\, dL \subset T^*T\mathcal{Q}$, $\mathrm{im}\, dH \subset T^*T^*\mathcal{Q}$ and $\mathrm{im} \, X_H \subset TT^*\mathcal{Q}$ define \emph{so-called} Lagrangian submanifolds \cite{Maslov65,Weinstein71} of the corresponding spaces which are related by Tulczyjew's isomorphisms. A similar analysis can be carried out in our setting, with $\mathrm{im} \, d \bar{\mathcal{H}}_{-1} \subset T^*T^*T\mathcal{Q}$, $\mathrm{im} \, d \tilde{\mathcal{L}} \subset T^*TT^*\mathcal{Q}$ and so on, but this is beyond the scope of this publication.
\end{remark}

To end this section, let us go back once more to the case of force-controlled Lagrangian and Hamiltonian systems. Notice that with all the structure introduced in the previous section, we can construct 
\begin{equation*}
    \begin{array}{rcccc}
         \tilde{\chi}_{\mathbb{F}L}^{1} = \tilde{\beta}_{\mathcal{Q}}^{\mathcal{F}}\circ (\tilde{\alpha}_{\mathcal{Q}}^{\mathcal{F}})^{-1} \circ \chi_{\mathbb{F}L}^{\oplus} &:& T^*T\mathcal{Q} \oplus_{T\mathcal{Q}}^{\alpha,L} \mathcal{E} &\to& T^*T^*\mathcal{Q} \oplus_{T^*\mathcal{Q}} \mathcal{F}\\
         && (q,\xi,\varpi_{q},\varpi_{\xi},u) &\to& (q, \varpi_{\xi},-\varpi_q,\chi(q,D_2 H(q, \varpi_{\xi}))u)
    \end{array}
\end{equation*}
This can be regarded as a partial cotangent lift in the first argument of the sum. This affords us the following analogue of Theorem \ref{thm:relation_hamiltonian_new_functions}.
\begin{theorem}
    Let $(\mathcal{Q},\mathcal{E},L,f_L^{\mathcal{E}})$ be a hyperregular force-controlled Lagrangian system with $\tilde{\mathcal{L}}_L^{\mathcal{E}}$, $\tilde{\mathcal{H}}_L^{\mathcal{E}}$, its new control Lagrangian and Hamiltonian respectively, for a running cost $C:\mathcal{E} \to \mathbb{R}$. Let $(T^*\mathcal{Q},\mathcal{F},H,f_L^{\mathcal{F}})$ be its corresponding associated force-controlled Hamiltonian system through a vector bundle morphism $\chi_{\mathbb{F}L}$ over $\mathbb{F}L$. Finally, let $C_H:\mathcal{F} \to \mathbb{R}$ be its associated running cost and $\mathcal{H}_{\lambda_0}: T^*T^*\mathcal{Q} \oplus_{T \mathcal{Q}} \mathcal{F} \to \mathbb{R}$ the Pontryagin's control Hamiltonian for the force-controlled Hamiltonian system. Then,
    \begin{itemize}
        \item $\tilde{\mathcal{L}}_L^{\mathcal{E}} = \hphantom{-}\mathcal{H}_{-1} \circ \tilde{\chi}_{\mathbb{F}L}^{1} \circ \mathbb{F}\tilde{\mathcal{L}}_{L}^{\mathcal{E}}$,
        \item $\tilde{\mathcal{H}}_L^{\mathcal{E}} = -\mathcal{H}_{-1} \circ \tilde{\chi}_{\mathbb{F}L}^{1}$.
    \end{itemize}
\end{theorem}
\begin{proof}
    This reduces to the computations performed at the end of Section \eqref{ssec:OptimalControlLagrangian} for the particular choice $\chi_{\mathbb{F}L}(q,v_q,u) = (q,D_2 L(q,v_q), u)$, but the computations for a general $\chi_{\mathbb{F}L}$ follow almost identically.
\end{proof}

Similar considerations as in the controlled SODE case apply in terms of presymplecticity in the case of explicit control dependence and symplecticity at local optima. Without going into much detail, the construction of the presymplectic forms in $T^*T^*\mathcal{Q} \oplus_{T^*\mathcal{Q}} \mathcal{F}$, $TT^*\mathcal{Q} \oplus_{T^*\mathcal{Q}} \mathcal{F}$ and $T^*T\mathcal{Q} \oplus_{T^*\mathcal{Q}}^{\tilde{\alpha}} \mathcal{F}$ follows trivially by pullback through the projections into the standard triple. Since $\chi_{\mathbb{F}L}$ is, by assumption and under the assumption of a hyperregular $L$, an isomorphism, it is a presymplectomorphism. By commutation of the diagram at the end of the previous section, it provides $T^*T\mathcal{Q} \oplus_{T\mathcal{Q}}^{\alpha, L}$ with the same presymplectic form as that obtained from $T^*T\mathcal{Q}$ by pullback through $\mathrm{pr}_1^{\alpha,L}$. Finally, this presymplectic form can be pulled back once more via $\mathbb{F}\tilde{\mathcal{L}}_L^{\mathcal{E}}$. Since $\tilde{\mathcal{L}}_L^{\mathcal{E}}$ is also hyperregular, the previous map is a diffeomorphism and the resulting Poincar{\'e}-Cartan 2-form is another presymplectomorphism. At optima this collapses to the same symplectic form on $TT\mathcal{Q}$ as that obtained by pullback through $\mathbb{F}\tilde{\mathcal{L}}_L$.


\section{Symmetries of the OCP and Noether’s theorem}
\label{sec:symmetries_noether}
While in the previous section we made an extensive analysis for the SODE case as well as for the force-controlled Lagrangian, respectively Hamiltonian, case, in this section we restrict ourselves on the SODE case. Analogous results can be obtained for the force-controlled Lagrangian, respectively Hamiltonian, case.

\subsection{Symmetries in optimal control problems for second-order systems}
Consider a left (right) action of a Lie group $\mathcal{G}$ on $\mathcal{Q}$ defined by $\Phi: \mathcal{G} \times \mathcal{Q} \rightarrow \mathcal{Q}$, we will also use the notation $\Phi_g(q)$. The action induces a tangent and a cotangent lift defined in local adapted coordinates as follows.
\begin{equation*}
    \begin{aligned}
        \Phi^{TQ} : \mathcal{G}\times T\mathcal{Q} &\rightarrow T\mathcal{Q} &\quad \Phi^{T^*Q} : \mathcal{G}\times T^*\mathcal{Q} &\rightarrow T^*\mathcal{Q} \\
        (g,(q,v_q)) & \mapsto (\Phi_g(q), D_q\Phi_g(q) \cdot v_q) & \quad (g,(q,\lambda_q)) & \mapsto (\Phi_g(q), \left(D_{q}\Phi_{g^{-1}}(\Phi_g(q))\right)^\top \cdot \lambda_q)
    \end{aligned}
\end{equation*}
In what follows, we will make extensive use the following
\begin{definition} Let $\mathcal{A}$, $\mathcal{B}$ be smooth manifolds and $\mathcal{G}$ a Lie group acting on these with actions $\Phi^\mathcal{A}$ and $\Phi^\mathcal{B}$ respectively.
    \begin{itemize}
        \item Let $\phi: \mathcal{A} \to \mathbb{R}$. $\phi$ is said to be invariant if $\phi(\Phi^\mathcal{A}_g(x)) = \phi(x)$, for all $g \in \mathcal{G}$, $x \in \mathcal{A}$.
        \item Let $\psi: \mathcal{A} \rightarrow \mathcal{B}$. The map $\psi$ is said to be equivariant if it satisfies $\psi(\Phi^{\mathcal{A}}_g(x)) = \Phi^{\mathcal{B}}_g(\psi(x))$, for all $g \in \mathcal{G}$, $x \in \mathcal{A}$.
    \end{itemize}
\end{definition}
 Notice that the actions lifted to the tangent and the cotangent bundles have the property that the pairing between vectors and covectors is invariant with respect to the induced actions, that is, $$\langle \left(D_{q}\Phi_{g^{-1}}(\Phi_g(q))\right)^* \lambda_q,  D_q\Phi_g(q) v_q\rangle = \langle \lambda_q, v_q\rangle.$$ The action can be further lifted to double bundles $T^*T\mathcal{Q}, \ TT^*\mathcal{Q}, \  T^*T^*\mathcal{Q}$. On $T^*T\mathcal{Q}$ it is defined by
\begin{equation*}
    \begin{aligned}
       \Phi^{T^*TQ} : \mathcal{G} \times T^*T\mathcal{Q}   \rightarrow&  T^*TQ \\
         (g,(q,v_q, \lambda_q, \lambda_v))  \mapsto& (\Phi_g^{TQ}(q,v_q), \left(D_{(q,v_q)}\Phi_{g^{-1}}^{TQ}(\Phi_g^{TQ}(q,v_q))\right)^\top \cdot (\lambda_q, \lambda_v)^\top ),
    \end{aligned}
\end{equation*}
where we have
\begin{multline} \label{eq:transform.lambdaq}
\left(D_{(q,v_q)}\Phi_{g^{-1}}^{TQ}(\Phi_g^{TQ}(q,v_q))\right)^\top \cdot (\lambda_q, \lambda_v)^\top = ( \left( D_q\Phi_{g^{-1}}(\Phi_g(q))\right)^\top \cdot \lambda_q  + \left( D^2_q\Phi_{g^{-1}}(\Phi_g(q)) \cdot D_q \Phi_g(q) \cdot v_q \right)^\top \cdot \lambda_v, \\ \left( D_q\Phi_{g^{-1}}(\Phi_g(q))\right)^\top \cdot \lambda_v ).    
\end{multline}
On $TT^*\mathcal{Q}$ the action takes the form
\begin{equation*}
    \begin{aligned}
         \Phi^{TT^*Q} : \mathcal{G} \times TT^*\mathcal{Q}   \rightarrow& TT^*\mathcal{Q} \\
         (g,(q,\kappa,v_q,v_{\kappa}))  \mapsto& (\Phi_g^{T^*Q}(q,\kappa), D_{(q,\kappa)}\Phi_g^{T^*Q}(q,\kappa) \cdot (v_q,v_{\kappa})^\top ),\\ 
    \end{aligned}
\end{equation*}
   with 
\begin{multline} \label{eq:transform.vq}
D_{(q,\kappa)}\Phi_g^{T^*Q}(q,\kappa) \cdot (v_q,v_{\kappa})^\top = ( D_q\Phi_g(q) \cdot v_q, \\ \left( D^2_q\Phi_{g^{-1}}(\Phi_g(q)) \cdot D_q \Phi_g(q) \cdot v_q \right)^\top \cdot \kappa + \left( D_q\Phi_{g^{-1}}(\Phi_g(q))\right)^\top \cdot v_\kappa ).    
\end{multline}
Finally, the action on $T^*T^*\mathcal{Q}$ is given by
\begin{equation*}
    \begin{aligned}
        \Phi^{T^*T^*Q} :  \mathcal{G} \times T^*T^*\mathcal{Q}    \rightarrow& T^*T^*\mathcal{Q} \\
        (g,(q,\kappa,p_q,p_{\kappa}))  \mapsto& (\Phi_g^{T^*Q}(q,\kappa), \left(D_{(q,\kappa)}\Phi_{g^{-1}}^{T^*Q}(\Phi_g^{T^*Q}(q,\kappa))\right)^\top \cdot (p_q,p_{\kappa})^\top ). 
    \end{aligned}
\end{equation*}
where 
\begin{multline} \label{eq:transform.pq}
\left(D_{(q,\kappa)}\Phi_{g^{-1}}^{T^*Q}(\Phi_g^{T^*Q}(q,\kappa))\right)^\top \cdot (p_q,p_{\kappa})^\top  = (\left( D_q\Phi_{g^{-1}}(\Phi_g(q))\right)^\top \cdot p_q \\ + \kappa^\top \cdot D_q\Phi_{g^{-1}}(\Phi_g(q)) \cdot D^2_q \Phi_g(q) \cdot D_q\Phi_{g^{-1}}(\Phi_g(q)) \cdot p_\kappa,  D_q \Phi_g(q) \cdot p_\kappa ).   
\end{multline}

\begin{theorem}\label{th:equivar.Tul}
The maps $\alpha_\mathcal{Q}: TT^*\mathcal{Q} \to T^*T\mathcal{Q}$ and $\beta_\mathcal{Q}: TT^*\mathcal{Q} \to T^*T^*\mathcal{Q}$ are equivariant with respect to the corresponding lifted groups actions.
\end{theorem}

\begin{proof}
    The equivariance of $\alpha_\mathcal{Q}$ can be easily checked by inspection of the action on $v_q$ in \eqref{eq:transform.vq} and on $(p_q,p_v)$ in \eqref{eq:transform.pq}. The equivariance of $\beta_{\mathcal{Q}}(q,\kappa,v_q,v_{\kappa}) = (q,\kappa,-v_{\kappa},v_q) = (q,\kappa, p_q, p_\kappa)$ follows from the observation
    \begin{multline*}
        \left( D^2_q\Phi_{g^{-1}}(\Phi_g(q)) \cdot D_q \Phi_g(q) \cdot v_q \right)^\top \cdot \kappa + \kappa^\top \cdot D_q\Phi_{g^{-1}}(\Phi_g(q)) \cdot D^2_q \Phi_g(q) \cdot D_q\Phi_{g^{-1}}(\Phi_g(q)) \cdot v_q  \\ = \kappa^\top \cdot D^2_q (\Phi_g \circ \Phi_{g^{-1}})(\Phi_g(q))\cdot v_q = 0. \qedhere
    \end{multline*} 
\end{proof}

Let us now define an action on the state-control space $\mathcal{E}$, which we require to be a vector bundle morphism over $\Phi^{TQ}$. In local adapted coordinates $(q, v_q, u)$, it takes the form
\begin{equation*}
        \Phi^{\mathcal{E}} : 
        (g,(q,v_q, u))  \mapsto (\Phi^{TQ}_g(q, v_q), \Psi_g(q, v_q, u)), 
\end{equation*}
where $\Psi_{g}: \mathcal{E} \to ((\pi^{\mathcal{E}})^{-1} \circ \Phi_g^{T\mathcal{Q}} \circ \pi^{\mathcal{E}})(\mathcal{E})$ defines the action on the control fiber. It is assumed to be linear in $u$ to preserve the vector bundle structure.
To analyze the symmetries of an optimal control problem of second-order systems, we need to introduce the lift of the action to the double bundles involved in Section~\ref{sec:tulcyjew_and_PMP}, namely $T^*TQ\oplus\mathcal{E}, \ TT^*\mathcal{Q} \oplus_{T\mathcal{Q}}^\alpha \mathcal{E}, \  T^*T^*\mathcal{Q} \oplus_{T\mathcal{Q}}^\beta \mathcal{E}$. We may uniquely define the corresponding lifted actions:
\begin{itemize}
    \item $\Phi^{T^*TQ\oplus\mathcal{E}} : \mathcal{G} \times T^*TQ\oplus_{T\mathcal{Q}}\mathcal{E} \to T^*TQ\oplus_{T\mathcal{Q}}\mathcal{E}$ such that $\mathrm{pr}_1$ is equivariant under $(\Phi^{T^*TQ\oplus\mathcal{E}},\Phi^{T^*TQ})$ and $\mathrm{pr}_2$ under $(\Phi^{T^*TQ\oplus\mathcal{E}},\Phi^{\mathcal{E}})$;
    \item $\Phi^{TT^*Q\oplus\mathcal{E}} : \mathcal{G} \times TT^*Q\oplus_{T\mathcal{Q}}^{\alpha}\mathcal{E} \to TT^*Q\oplus_{T\mathcal{Q}}^{\alpha}\mathcal{E}$ such that $\mathrm{pr}_1^{\alpha}$ is equivariant under $(\Phi^{TT^*Q\oplus\mathcal{E}},\Phi^{TT^*Q})$ and $\mathrm{pr}_2^{\alpha}$ under $(\Phi^{TT^*Q\oplus\mathcal{E}},\Phi^{\mathcal{E}})$;
    \item $\Phi^{T^*T^*Q\oplus\mathcal{E}} : \mathcal{G} \times T^*T^*Q\oplus_{T\mathcal{Q}}^{\beta}\mathcal{E} \to T^*T^*Q\oplus_{T\mathcal{Q}}^{\beta}\mathcal{E}$ such that $\mathrm{pr}_1^{\beta}$ is equivariant under $(\Phi^{T^*T^*Q\oplus\mathcal{E}},\Phi^{T^*T^*Q})$ and $\mathrm{pr}_2^{\beta}$ under $(\Phi^{T^*T^*Q\oplus\mathcal{E}},\Phi^{\mathcal{E}})$.
\end{itemize}
By construction, from the commutativity of the extended Tulczyjew triple and the result of Theorem \ref{th:equivar.Tul}, we get the following
\begin{theorem} \label{th:equivar.Tul.ext}
    The maps $\alpha_\mathcal{Q}^{\mathcal{E}}: TT^*\mathcal{Q} \oplus_{T\mathcal{Q}}^\alpha \mathcal{E} \to T^*T\mathcal{Q} \oplus_{T\mathcal{Q}} \mathcal{E}$ and $\beta_\mathcal{Q}^{\mathcal{E}}: TT^*\mathcal{Q} \oplus_{T\mathcal{Q}}^\alpha \mathcal{E} \to T^*T^*\mathcal{Q} \oplus_{T\mathcal{Q}}^{\beta} \mathcal{E}$ are equivariant with respect to the corresponding lifted groups actions.
\end{theorem}
\begin{theorem} \label{th:equiv.invariances}
    The following statements are equivalent.
    \begin{enumerate}
        \item[a)] \label{it:L}$\tilde{\mathcal{L}}^\mathcal{E}$ is invariant with respect to the action by $\Phi^{TT^*Q\oplus\mathcal{E}}_g$;
        \item[b)] \label{it:H}$\tilde{\mathcal{H}}^\mathcal{E}$ is invariant with respect to the action by $\Phi^{T^*T^*Q\oplus\mathcal{E}}_g$;
        \item[c)] \label{it:PMP} $\mathcal{H}_{-1}$ is invariant with respect to the action by $\Phi^{T^*TQ\oplus\mathcal{E}}_g$.
    \end{enumerate}
\end{theorem}
\begin{proof} 
    The equivalence of a) and b) follows from the equivariance of $\beta_\mathcal{Q}^{\mathcal{E}}$ and 
   the equivalence between a) and c) follows from the equivariance of $\alpha_{\mathcal{Q}}^{\mathcal{E}}$.
\end{proof}
To apply the results of Theorem~\ref{th:equiv.invariances} in the context of OCPs, we provide the following
\begin{definition}
 An OCP is $\mathcal{G}$-symmetric (equivalently, admits a $\mathcal{G}$-symmetry) with respect to the Lie group action $\Phi^{\mathcal{E}}$, if the controlled system is equivariant and the running cost is invariant with respect to the corresponding actions.
\end{definition}
\begin{theorem} \label{th:OCP.symmetric}
    If an OCP of a second-order system is $\mathcal{G}$-symmetric, then $\tilde{\mathcal{L}}^\mathcal{E}$ is invariant with respect to the action by $\Phi^{TT^*Q\oplus\mathcal{E}}_g$.
\end{theorem}
\begin{proof} 
It follows from \cite{Ohsawa2013} that if an optimal control problem is $\mathcal{G}$-symmetric, then  $\mathcal{H}_{-1}$ is invariant with respect to the action of $\Phi^{T^*TQ\oplus\mathcal{E}}_g$. Using the equivalence of a) and c) in Theorem~\ref{th:equiv.invariances} we conclude that $\tilde{\mathcal{L}}^\mathcal{E}$ is invariant with respect to the action by $\Phi^{TT^*Q\oplus\mathcal{E}}_g$.
\end{proof}
\begin{remark}
    The results of Theorem~\ref{th:equiv.invariances} and Theorem~\ref{th:OCP.symmetric} also hold for the case of a regular force-controlled Lagrangian system. The adaptation of the proofs is straightforward.
\end{remark}
\subsection{Noether’s theorem}
When an optimal control problem is symmetric with respect to a group action, Noether’s theorem adapted to the OCP setting permits us to describe conserved quantities along the optimal solutions. Let us consider a one-parameter group of transformations $\mathcal{G}_s, \ s \in \mathbb{R}$ and the associated action $\Phi_s$. Let us denote the infinitesimal generator of $\Phi_s$ by $X^\mathcal{Q} \in \mathfrak{X}(\mathcal{Q})$ defined by
$$
X^\mathcal{Q}(q) = \left. \frac{\partial \Phi_s(q)}{\partial s}\right|_{s=0}. 
$$ As $\Phi_s$ can be lifted to an action on one of the spaces defined previously, we use the upper index of $X^\mathcal{Q}$ to denote the space, on which the infinitesimal generator is defined. Based on \cite{Torres2002}, Noether's theorem in the context of second order systems can be stated as follows.
\begin{theorem} \label{th:Noether.OCP}
    If an \ref{eq:OCP} admits a symmetry  with respect to the action of $\Phi_s^{\mathcal{E}}$, then the following momentum map is conserved along optimal solutions
    $$ I(q, v, \lambda_q, \lambda_{v}) = \lambda^\top X^{T\mathcal{Q}}(q, v) = \left. \lambda_q^\top \frac{\partial \Phi_s(q)}{\partial s}\right|_{s=0} + \left. \lambda_{v}^\top \frac{\partial T_q\Phi_s(v)}{\partial s}\right|_{s=0}. $$
\end{theorem}
As we can see, the momentum map depends on the adjoint variable and Noether’s theorem rather describes the first integrals of the extremal flow of the state-adjoint system. In the setting of the new Lagrangian formulation, Noether's theorem takes the following form. 
\begin{theorem} \label{th:Noether.new.Lagrangian}
    If the optimal Lagrangian $\tilde{\mathcal{L}}^\mathcal{E}$ is invariant with respect to the action of $\Phi^{TT^*Q\oplus\mathcal{E}}_s$, then the following momentum map is conserved along optimal solutions
    \begin{equation} \label{eq:momentum.new.Lagrangian}
    \begin{aligned}
        I_{\tilde{\mathcal{L}}^\mathcal{E}}(y, \dot y, u) &= \frac{\partial \tilde{\mathcal{L}}^\mathcal{E}(y, \dot y, u)}{\partial \dot y} X^{T^*\mathcal{Q}}(y) \\ &= \left( \dot{\kappa} +  D_2 X_v(q, \dot q, u)^\top \kappa - D_2 C(q, \dot q, u)\right)^\top  \left. \frac{\partial \Phi_s(q)}{\partial s}\right|_{s=0} + \kappa^\top \left. \frac{\partial T_q\Phi_s(\dot q)}{\partial s} \right|_{s=0}.
    \end{aligned}
    \end{equation} 
\end{theorem}
\begin{proof}
Let us fix $t \in (0,T)$. Let $(y,\dot y, u)$ be a solution of Euler-Lagrange equations associated with $\tilde{\mathcal{L}}^\mathcal{E}$. We denote $(y_s, \dot y_s, u_s) =  \Phi^{TT^*Q\oplus\mathcal{E}}_s(y, \dot y, u) $ the transformed solutions. The invariance of $\tilde{\mathcal{L}}^\mathcal{E}$ with respect to $\Phi^{TT^*Q\oplus\mathcal{E}}_s$ implies the invariance of the following action integral
$$\int_0^t \tilde{\mathcal{L}}^\mathcal{E}(y_s(\tau), \dot y_s(\tau), u_s(\tau)) d\tau.$$
This implies in particular
$$
\begin{aligned}
    0 &= \left. \frac{d}{d s}\int_0^t \tilde{\mathcal{L}}^\mathcal{E}(y_s(\tau), \dot y_s(\tau), u_s(\tau)) d\tau \right|_{s=0} \\
    & = \left. \int_0^t \left[\left( \frac{\partial \tilde{\mathcal{L}}^\mathcal{E}}{\partial y}  - \frac{d}{ dt} \frac{\partial \tilde{\mathcal{L}}^\mathcal{E}}{\partial \dot y}\right) \frac{ \partial y_s(\tau)}{\partial s} + \frac{\partial \tilde{\mathcal{L}}^\mathcal{E}}{\partial u}\frac{ \partial u_s(\tau)}{\partial s} \right] d\tau + \left. \frac{\partial \tilde{\mathcal{L}}^\mathcal{E}}{\partial \dot y} \frac{ \partial y_s(\tau)}{\partial s}\right|^t_0 \ \right|_{s=0}\\
    & = \frac{\partial \tilde{\mathcal{L}}^\mathcal{E}(y(t), \dot y(t), u(t))}{\partial \dot y} \left. \frac{ \partial y_s(t)}{\partial s}  \right|_{s=0} - \frac{\partial \tilde{\mathcal{L}}^\mathcal{E}(y(0), \dot y(0), u(0))}{\partial \dot y} \left. \frac{ \partial y_s(0)}{\partial s}  \right|_{s=0},
\end{aligned}
$$
which finishes the proof.
\end{proof}
\begin{remark}
    The momentum \eqref{eq:momentum.new.Lagrangian} admits a simple expression in the Hamiltonian framework. Applying the Legendre transform $\mathbb{F}\tilde{\mathcal{L}}^{\mathcal{E}}$, we get
    $$I_{\tilde{\mathcal{H}}^{\mathcal{E}}}(y, p_y) = p_y^\top X^{T^*\mathcal{Q}}(y) = p_q^\top  \left. \frac{\partial \Phi_s(q)}{\partial s}\right|_{s=0} + p_\lambda^\top \left. \frac{\partial (T_q\Phi_s)^*(\lambda)}{\partial s} \right|_{s=0}. $$
\end{remark}
It is easy to see that both conservation laws described in Theorem~\ref{th:Noether.OCP} and Theorem~\ref{th:Noether.new.Lagrangian} coincide and the relation is defined by $\lambda_q = v_\kappa +  D_2 X_v(q, \dot q, u)^\top \kappa - D_2 C(q, \dot q, u)$ and $\lambda_q = \kappa$.

Let us consider the controlled Lagrangian system \eqref{eq:force_controlled_EL}. Noether's theorem for controlled Lagrangian systems can be applied in this case \cite{MarsdenWest01}. 
\begin{theorem} \label{th:Noether.Lagrang.syst} If a Lagrangian $L: T\mathcal{Q} \to \mathbb{R}$ is invariant with respect to $\Phi^{TQ}_s$ and the force $f_L^{\mathcal{E}}$ is orthogonal to the one-parameter group of transformations, i.e.
$$f_L^{\mathcal{E}}(q(t),\dot{q}(t),u(t)) \, \left. \frac{\partial \Phi_s(q)}{\partial s} \right|_{s=0} = 0 \qquad \text{for any } (q, \dot q, u) \in \mathcal{E},$$
then the associated momentum map $I_L$ is a conserved quantity along any admissible trajectory of the controlled Lagrangian system
$$ I_L(q, \dot q) = \frac{\partial L(q, \dot q)}{\partial \dot q} X^{\mathcal{Q}}(q). $$ 
\end{theorem}
Assume that Noether's theorem for controlled Lagrangian systems applies. This implies that the control system given by \eqref{eq:force_controlled_EL} is equivariant with respect to the action $\Phi^\mathcal{E}_s$. Assume in addition that the running cost $C$ is invariant with respect to $\Phi^\mathcal{E}_s$. This implies by Theorem~\ref{th:OCP.symmetric} that $\tilde{\mathcal{L}}_L^\mathcal{E}$ is invariant with respect to the action by $\Phi^\mathcal{E}_s$ and Theorem~\ref{th:Noether.new.Lagrangian} can be applied. The associated momentum map is given by
$$ 
\begin{aligned}
I_{\tilde{\mathcal{L}}_L^\mathcal{E}}(y, \dot y, u) &= \frac{\partial \tilde{\mathcal{L}}_L^\mathcal{E}(y, \dot y, u)}{\partial \dot y} X^{T\mathcal{Q}}(y) \\
& =  \left(D_{22} L(q,\dot q) \, v_{\xi} + \left[ D_{21} L(q,\dot q) + D_2 f_L^{\mathcal{E}}(q,\dot q,u) \right] \, \xi - D_2 C(q,\dot q,u) \right) \left. \frac{\partial \Phi_s(q)}{\partial s}\right|_{s=0}\\  
& +  \frac{\partial L(q, \dot q)}{\partial \dot q} \left. \frac{\partial T_q\Phi_s(\xi)}{\partial s} \right|_{s=0}
\end{aligned}
$$
Applying the Legendre transform, we can express the momentum maps of the Lagrangian system as follows.
$$ I_L(q, \dot q) = p_q X^{\mathcal{Q}}(q) , \qquad I_{\tilde{\mathcal{L}}_L^\mathcal{E}}(y, \dot y, u) = \varpi_q  \left. \frac{\partial \Phi_s(q)}{\partial s}\right|_{s=0}  + p_q \left. \frac{\partial T_q\Phi_s(\xi)}{\partial s} \right|_{s=0}.$$ 
Notice that both $I_L$ and $I_{\tilde{\mathcal{L}}_L^\mathcal{E}}$ are conserved quantities of the optimal control problem and in addition $I_L$ and $I_{\tilde{\mathcal{L}}_L^\mathcal{E}}$ are functionally independent.\\

Noether's theorem can also be approached using the Hamiltonian formulation of the necessary conditions for optimality. The control Hamiltonian $\mathcal{H}_{-1}$ is defined on $T^* \mathcal{M} \oplus_\mathcal{M} \mathcal{E}$, which has a natural pre-symplectic structure. Noether's theorem of Hamiltonian systems on pre-symplectic manifolds was considered in \cite{Ciaglia2022}. Another approach adapted in \cite{Schaft1987} treats implicit Hamiltonian systems based on port-Hamiltonian formalism. These approaches lead to equivalent conservation laws as those obtained by our variational approach in Theorem~\ref{th:Noether.OCP} and Theorem~\ref{th:Noether.new.Lagrangian}. Notice, that we have restricted ourselves to symmetries obtained by the lifted Lie group actions. In general, a Hamiltonian $\mathcal{H}_{-1}$ can admit symmetries which are not generated by lifted actions. This case was considered in \cite{Blankenstein01} and such symmetries are called generalized symmetries of the OCP. As in Theorem~\ref{th:equiv.invariances}, if $\mathcal{H}_{-1}$ is invariant with respect to a generalized symmetry and $\alpha_{\mathcal{Q}}^{\mathcal{E}}$ is equivariant, then $\tilde{\mathcal{L}}^\mathcal{E}$ is invariant and Theorem~\ref{th:Noether.new.Lagrangian} still applies.

\section{Conclusions and future work}

In this work, we have generalized and analyzed in depth the new Lagrangian approach for the optimal control of second-order systems proposed in \cite{Leyendecker2024new}. In that article, the setting was restricted to a particular subset of optimal control problems, namely, those with cost functions quadratic in the controls and affine-controlled SODEs. The theory presented here now expands that setting to accommodate arbitrary cost functions and controlled SODEs. An extensive analysis of the geometric setting of this new approach has been performed, linking it to the original PMP through an extension of Tulczyjew's triple to accommodate the controls.\\

Our approach is rooted in the calculus of variations, and as such, we ascribe to a certain amount of analytic and algebraic regularity. In particular, we have introduced some definitions to frame the algebraic regularity assumed in this work. However, this is mainly done for simplicity's sake. The analysis performed shows that this approach and PMP in the same setting are one and the same, simply expressed in different spaces. This points to this new approach being applicable in a wider, less smooth setting. In particular, one may weaken the analytic regularity requirements, working with needle variations \cite{McShane39,Liberzon12}, making it possible to tackle algebraically singular problems. Abnormal multipliers may also be considered, though that may lead to a Routhian, i.e. a Lagrangian and Hamiltonian hybrid, approach.\\

We compared our approach to another with a similar objective, based on the reformulation of the OCP as higher-order Lagrangians \cite{Colombo2010,Colombo2016}. While said approach is very interesting and applicable to systems of arbitrary order, we were able to show that it is rather restrictive in its algebraic regularity requirements, somewhat limiting its applicability.\\

Force-controlled Euler-Lagrange equations have been studied as well, being one of our primary motivations to propose this approach. While the required changes are not fundamental, they significantly alter the geometry of the problem. In \cite{Konopik25a}, the numerical application and analysis of the approach proposed in \cite{Leyendecker2024new} was performed with very good results at the level of the OCP. However, it was also observed that, when working with force-controlled Euler-Lagrange equations as controlled SODEs, the resulting methods for the separate state and adjoint dynamics were not necessarily ``symplectic''. 
Actually, it can be shown that they are each symplectic but with respect to a symplectic form that does not necessarily coincide with the Poincar{\'e}-Cartan 2-form associated to the original Lagrangian of the mechanical system. 
We are confident that through the discretization of $\tilde{\mathcal{L}}_L^{\mathcal{E}}$ and $\tilde{\mathcal{L}}_L$ instead, the issue will be resolved. We will present results in that regard in an upcoming work.\\

Besides this previous point, in this work we have also shed some light on the role of the boundary terms that appear in the process. This led us to the realization that the discrete new Lagrangian approach can be viewed as a transformation of the discrete OCP, expressible in terms of generating functions of the first kind, as generating functions of mixed kind, first in positions, fourth in velocities.\\

Finally, we have also studied the symmetries of the optimal control problem in relation to the symmetries of the new Lagrangian formulation. In particular, we have shown that the new Lagrangian inherits the symmetries of the control Hamiltonian $\mathcal{H}_{-1}$. This will allow us to ensure preservation properties of the numerical methods based on the new Lagrangian and a discrete variational formulation. In addition, we have proven Noether's theorem based on the new formulation, which leads to the same conserved quantities as the well known results in optimal control. 
The question of reduction in this context is an interesting one to tackle in the future. The simpler case of systems on Lie groups may be easily handled and can be of great importance in applications such as multibody systems.\\




\section*{Acknowledgements}

The authors acknowledge the support of Deutsche Forschungsgemeinschaft (DFG) with the projects: LE 1841/12-1, AOBJ: 692092 and OB 368/5-1, AOBJ: 692093.







\printbibliography

\end{document}